\documentclass[12pt, reqno]{amsart}
\usepackage{amsmath,amssymb,amsthm,amstext}
\usepackage{enumerate}
\usepackage{graphicx}
\usepackage{amsfonts}
\usepackage{hyperref}
\usepackage{epsfig}
\usepackage{epstopdf}
\usepackage{caption}
\usepackage[utf8]{inputenc}
\usepackage[T1]{fontenc}
\usepackage{tcolorbox}
\usepackage{dutchcal}
\usepackage{mathtools}
\usepackage{xcolor}
\usepackage[margin=1 in]{geometry}
\usepackage{subcaption}
\usepackage{comment}
\usepackage{float}
\usepackage[section]{placeins}
\usepackage{chngcntr}
\usepackage{mathrsfs}
\usepackage{soul}
\usepackage{tikz}
\numberwithin{equation}{section}

\newtheorem{thm}{Theorem}[section]

\newtheorem{rem}{Remark}[section]
\newtheorem{cor}[thm]{Corollary}

\newtheorem{lem}{Lemma}[section]

\newcommand{\la}{\lambda}

\newcommand{\Om}{\Omega}

\newcommand{\p}{\partial}

\title{Numerical Approximation and Bifurcation Results for an Elliptic Problem with Superlinear Subcritical Nonlinearity on the Boundary}
\author{Shalmali Bandyopadhyay, Thomas Lewis, Dustin Nichols}

\begin{document}


\begin{abstract}  We develop numerical algorithms to approximate positive solutions of elliptic boundary value problems with superlinear subcritical nonlinearity on the boundary of the form $-\Delta u + u = 0$ in $\Omega$ with $\frac{\partial u}{\partial \eta} = \lambda f(u)$ on $\partial\Omega$ as well as an extension to a corresponding system of equations. While existence, uniqueness, nonexistence, and multiplicity results for such problems are well-established, their numerical treatment presents computational challenges due to the absence of comparison principles and complex bifurcation phenomena. We present finite difference formulations for both single equations and coupled systems with cross-coupling boundary conditions, establishing admissibility results for the finite difference method. We derive principal eigenvalue analysis for the linearized problems to determine unique bifurcation points from trivial solutions. The eigenvalue analysis provides additional insight into the theoretical properties of the problem while also providing intuition for computing approximate solutions based on the proposed finite difference formulation.  We combine our finite difference methods with continuation methods to trace complete bifurcation curves, validating established existence and uniqueness results and consistent with the results of the principle eigenvalue analysis.   
\end{abstract}
\maketitle 

\noindent
\textbf{Keywords:} Elliptic equations; Superlinear subcritical nonlinearity; Nonlinear boundary conditions; Bifurcation theory; Finite difference methods\\
\noindent \textbf{2020 Mathematics Subject Classification:} Primary: 35J66, 65N06; Secondary: 35B32, 65P30
\markboth{SHALMALI BANDYOPADHYAY, THOMAS LEWIS, DUSTIN NICHOLS}{SHALMALI BANDYOPADHYAY, THOMAS LEWIS, DUSTIN NICHOLS}
\section{Introduction} 

Elliptic partial differential equations play a crucial role in modeling various phenomena, including chemical reactions, ecological systems, population dynamics, and combustion theory \cite{CantrellCosner_2006, LaceyOckendonSabina_1998, Pao_book}. While extensive research has focused on elliptic equations with linear boundary conditions, such as Dirichlet, Neumann, and Robin conditions \cite{Dil_Oth_1999, Gurung-Gokul-Adhikary, Liu-Shi_2018}, certain scenarios where chemical reactions, biological bonding, or species interactions may occur in a narrow layer or region near the boundary necessitate the study of elliptic equations coupled with nonlinear boundary conditions \cite{Amann_1976, CantrellCosner_2006, Gordon-Ko-Shivaji, Inkmann, Liu-Shi_2018, Mavinga-Pardo_PRSE_2017, Amster-DeNapoli, Umezu-2023, BCDMP_2022}.
The mathematical investigation of such problems has attracted considerable attention due to their rich bifurcation structure and multiplicity phenomena. Existence results with parameter dependence for sublinear power nonlinearities have been established \cite{Mavinga-Pardo_PRSE_2017, Umezu-2023}, while superlinear cases with nonlinear boundary conditions have been studied using various analytical techniques \cite{Inkmann, BCDMP_2022}. The combination of variational methods \cite{Amster-DeNapoli}, bifurcation theory \cite{Liu-Shi_2018, Mavinga-Pardo_PRSE_2017}, and topological degree methods \cite{Gordon-Ko-Shivaji} has proven particularly effective in understanding the solution structure of superlinear problems.

In \cite{BCDMP_2024}, the authors studied the following elliptic problem with nonlinearity on the boundary:
\begin{equation}
\label{pde_lambda}
\left\{
\begin{array}{rcll}
-\Delta u +u &=&  0 \quad &\mbox{in}\quad \Om\,,\\
\frac{\p u}{\p \widehat{n}} &=& \la f(u)\quad &\mbox{on}\quad \p \Om\,,
\end{array}
\right. 
\end{equation}
where $\Omega \subset \mathbb{R}^n$ is bounded, $\lambda>0$ is a parameter, $\widehat{n}$ is the unit outward normal vector, and $f:[0, \infty) \to [0, \infty)$ is a locally Lipschitz continuous function that satisfies superlinear and subcritical growth conditions, i.e., there exists a constant $b>0$ such that
\begin{equation*}
\hspace{-1cm} 
\mathrm{(H)}_\infty \qquad\lim\limits_{s \to +\infty}\frac{f(s)}{s^p}=b
\quad \mbox{ with }\quad \begin{cases}
1 < p < \frac{N}{N-2} \; &\mbox{ if }\;N\ge 3,\\ 
p>1\; &\mbox{ if }\;N=2   \,.
\end{cases}
\end{equation*}
The authors proved existence and multiplicity of positive solutions using rescaling methods combined with degree theory and bifurcation theory. Observe that, for $N = 1$, the subcritical condition is identical as the $N=2$ case: both require only $p > 1$ with no upper bound restriction, unlike $N \geq 3$ which imposes the additional constraint $p < \frac{N}{N-2}$.
In the companion systems paper \cite{BCDMP_Systems}, the authors extended their analysis to the following coupled system:
\begin{equation}
	\label{PDE_Sys}
	\left\{
	\begin{array}{rcll}
		-\Delta u +u&=&0   \quad & \mbox{in}\quad \Omega\,,\\
		-\Delta v +v &=&0  \quad &\mbox{in}\quad \Omega\,,\\
		\frac{\partial u}{\partial \widehat{n}} &=&  \lambda f(v)\quad &\mbox{on}\quad \partial \Omega\,,\\
		\frac{\partial v}{\partial \widehat{n}} &=&  \lambda g(u)\quad &\mbox{on}\quad \partial \Omega\,.
	\end{array}
	\right. 
\end{equation}
The theoretical results in these papers establish several key findings. Under superlinear subcritical conditions, there exists a connected branch of positive solutions bifurcating from infinity when $\lambda \to 0^+$ (Theorem 1.1 in \cite{BCDMP_2024} and \cite{BCDMP_Systems}). When additional conditions near zero are satisfied (specifically $f(0) = 0$, $f'(0) > 0$ for single equations and similar conditions for systems), there exists a global connected branch of positive solutions bifurcating from the trivial solution at specific parameter values (Theorem 1.2 in both papers). Furthermore, the sign of higher-order terms determines whether the bifurcation is subcritical (left) or supercritical (right), leading to different solution multiplicity patterns (Theorem 1.3 in \cite{BCDMP_Systems}). However, in most scenarios, closed-form solutions of PDEs are not known or do not exist. Therefore, to visualize solutions when they exist and validate the theoretical predictions, we utilize numerical methods. Some numerical approximation techniques for reaction-diffusion equations can be found in \cite{Pao1, Pao2, Pao3, Neu-Swift, Neu-Swift-Sie}.  
We will formulate, analyze, and test finite difference methods for directly approximating solutions to \eqref{pde_lambda} and \eqref{PDE_Sys} in this paper while also using the methods to construct the approximate bifurcation diagrams that illustrate the various theoretical results of \cite{BCDMP_2024,BCDMP_Systems}. 

While extensive theoretical literature exists for superlinear elliptic boundary value problems, their numerical treatment has received significantly less attention due to inherent computational challenges that distinguish them from their sublinear counterparts. For sublinear problems, classical monotone iteration methods utilizing sub- and supersolution techniques provide robust computational frameworks where admissibility and convergence is guaranteed through the ordered structure of solutions \cite{Pao_book,Sattinger1973,Lewis_Morris_Zhang_2022}. However, superlinear problems present fundamental difficulties: the lack of a comparison principles makes traditional monotone iteration schemes ineffective, solution multiplicity and bifurcation phenomena require sophisticated continuation methods to capture all solution branches, and the rapid growth of nonlinearities can lead to numerical instabilities near bifurcation points \cite{BCDMP_2024,Mavinga-Pardo_PRSE_2017}. The monotone method, which relies on the existence of ordered sub- and supersolutions, becomes inapplicable when the superlinear growth prevents the construction of suitable comparison functions. Consequently, most existing rigorous numerical studies focus on sublinear cases where the method of upper and lower solutions and monotone iterative techniques can be successfully applied, leaving a significant gap in the computational validation of theoretical predictions for superlinear problems. Our work partially addresses this computational void by developing finite difference algorithms along with a modified sub- and supersolution technique for proving admissibility that can be combined with continuation methods to generate bifurcation diagrams that validate the rich theoretical structure predicted for superlinear elliptic problems with nonlinear boundary conditions. The numerical tests in Sections \ref{subsec:connection-single} and \ref{subsec:connection-systems} illustrate the analytical PDE results.  The numerical analysis and solver strategy rely heavily upon the theoretical results for the underlying problems.

In this manuscript, we develop numerical algorithms for superlinear problems \eqref{pde_lambda} and \eqref{PDE_Sys} using finite difference methods.  The methods will be combined with continuation techniques to validate theoretical results in the one-dimensional case through a series of numerical tests. Our approach addresses fundamental computational challenges through several key innovations: (i) we perform a detailed eigenvalue analysis of the corresponding linearized problems to compute precise bifurcation points that can help construct sufficient initial guesses for solving the discrete nonlinear system of equations and generating bifurcation curves, (ii) we introduce a novel cutoff function technique to address numerical challenges associated with the admissibility analysis for superlinear boundary conditions, and (iii) we implement robust continuation strategies that successfully capture complete bifurcation structures including both subcritical and supercritical bifurcation (see \cite[Thm. 4.4]{BCDMP_2024}). We provide numerical validation for both single equations and coupled systems, demonstrating the transition from subcritical to supercritical bifurcations and confirming theoretical predictions regarding solution multiplicity and global bifurcation structure. Our computational framework overcomes some of the inherent numerical difficulties associated with superlinear problems while providing quantitative validation of the theoretical predictions established in \cite{BCDMP_2024, BCDMP_Systems}.  The principle eigenvalue analysis also helps benchmark the performance of our methods for approximating the bifurcation point.

The remainder of this paper is organized as follows. In Section \ref{sec:FD}, we establish the finite difference formulation and discrete operators used to approximate the differential equations and boundary conditions.  We prove admissibility results for the finite difference method and derive qualitative properties of the discrete solutions. Section \ref{sec:single-equation} focuses on the single equation case, where we present the coding algorithm for computing the one-dimensional bifurcation diagrams, derive the principal eigenvalue for the linearized problem, and compute bifurcation diagrams that validate the analytical predictions from \cite{BCDMP_2024} regarding bifurcation from infinity and multiplicity of positive solutions. Section \ref{sec:coupled-systems} extends our analysis to the coupled systems case, developing the corresponding numerical framework with principal eigenvalue analysis for the linearized system and computing bifurcation diagrams that confirm the richer bifurcation structure predicted by the analysis in \cite{BCDMP_Systems}.

\section{Finite Difference Approximation}
\label{sec:FD}
In this section we formulate our finite difference (FD) methods for approximating solutions to \eqref{pde_lambda} and derive various properties.  
The main idea in the FD formulation is to approximate all differential operators by discrete operators using difference quotients. We first discretize the domain.  Then, at each of the grid points, we approximate the value of the solution by solving the algebraic system of equations that results from replacing the differential operators with discrete difference operators.
The grid functions produced by the FD method serve as approximations for the underlying PDE solutions.  

\section{Finite Difference Approximation}
\label{sec:FD}
In this section we formulate our finite difference (FD) methods for approximating solutions to \eqref{pde_lambda} and derive various properties.  
The main idea in the FD formulation is to approximate all differential operators by discrete operators using difference quotients. We first discretize the domain.  Then, at each of the grid points, we approximate the value of the solution by solving the algebraic system of equations that results from replacing the differential operators with discrete difference operators.
The grid functions produced by the FD method serve as approximations for the underlying PDE solutions.  

\subsection{Difference Operators and Notation}
\label{sec:FD_operators}
Assume the domain $\Omega \subset \mathbb{R}^N$ is an $N$-rectangle. In other words, $\Omega = (a_1,b_1)\times (a_2,b_2)\times \cdots (a_N,b_N)$. Let $m_i \geq 4$ be a positive integer and $h_i=\frac{b_i-a_i}{m_i-1}$ for $i=1,2,\ldots, N$. Define ${h}=(h_1,h_2,\cdots, h_N) \in \mathbb{R}^N$, $M=\prod_{i=1}^{N}(m_i)$, and $\mathbb{N}_M^N=\{{\alpha}=(\alpha_1,\alpha_2, \ldots \alpha_N)\mid 1\leq \alpha_i\leq m_i, i=1,2, \ldots, N\}$. Next we partition $\Omega$ into $\prod_{i=1}^N(m_i-1)$ sub-$N$ rectangles with grid points ${x_{\alpha}}=(a_1+(\alpha_1-1)h_1, a_2+(\alpha_2-1)h_2, \cdots, a_N+(\alpha_N-1)h_N)$ for each multi-index $\alpha \in \mathbb{N}_M^N$. We call $\mathcal{T}_{{h}}=\{{x}_{\alpha}\}_{\alpha \in \mathbb{N}_M^N}$ a grid for $\overline{\Omega}$.
We let $h^* = \max_{i=1}^N h_i$ and $h_* = \min_{i=1}^N h_i$. 

Let $\{e_i\}_{i=1}^N$ denote the canonical basis vectors for $\mathbb{R}^N$. We define the discrete operators for approximating first order partial derivatives 
$\frac{\partial}{\partial x_i} u(x)$ by
\begin{equation}\label{cdiff}
\left\{ \begin{array}{rcll}
\delta^+_{x_i,h_i}u({x}) & := &\frac{u({x}+h_ie_i)-u({x})}{h_i},\\
\delta^-_{x_i,h_i}u({x}) &:= &\frac{u({x})-u({x}-h_ie_i)}{h_i},\\
\delta_{x_i,h_i}u({x}) &:= &\frac12 \delta^+_{x_i, h_i} u(x) + \frac12 \delta^-_{x_i, h_i} u(x) & = \frac{u({x}+h_i e_i)-u({x}-h_ie_i)}{2h_i}
 \end{array}\right.
\end{equation} 
for the function $u:\mathbb{R}^N \to \mathbb{R}$ and 
\begin{equation}\label{grid}
\left\{ \begin{array}{rcl}
\delta^+_{x_i,h_i}u_{h}(x_{\alpha})
&:=&\frac{u_{h}(x_{\alpha+e_i})-u_{h}(x_{\alpha})}{h_i},\\
\delta^-_{x_i,h_i}u_{h}(x_{\alpha}) &:=&\frac{u_{h}(x_{\alpha})-u_{h}(x_{\alpha-e_i})}{h_i},\\
\delta_{x_i,h_i}u_h({x_{\alpha}}) &:= &\frac12 \delta^+_{x_i, h_i} u_h(x_{\alpha}) + \frac12 \delta^-_{x_i, h_i} u_h(x_{\alpha}) \\ &\,=& \frac{u_{h}(x_{\alpha+e_i})-u_{h}(x_{\alpha-e_i})}{2h_i}
 \end{array}\right. 
\end{equation}  for all $x_{\alpha} \in \mathcal{T}_h \cap \Omega$ for the grid function $u_h:\mathcal{T}_h\to \mathbb{R}$. Note that the discrete operators $\delta^\pm_{x_i,h_i}$ are first-order accurate whereas $\delta_{x_i,h_i}$ is second-order accurate. 
We also define the corresponding discrete gradient operators 
\[
[\nabla^\pm_h]_i  := \delta^\pm_{x_i, h_i}  , \qquad 
[\nabla_h]_i := \delta_{x_i, h_i}.
\]
Let $\widetilde{\partial \Omega} \subset \partial \Omega $  such that $\widetilde{\partial \Omega}:= \p\Om \setminus \{\mbox{the points where }\partial \Om \mbox{ is not smooth}\}$. For $x_\alpha \in \mathcal{T}_h \cap \widetilde{\partial \Omega}$, we define the discrete outward normal derivative operator using the discrete gradient operator $\nabla^*_{h}$ defined by 
\[
[\nabla^*_h u(x_\alpha)]_i \cdot \widehat{n}(x_\alpha) := \begin{cases}
\delta_{x_i, h_i}^+ u(x_\alpha)& \text{if } \widehat{n}_i(x_\alpha) < 0 , \\ 
\delta_{x_i, h_i}^- u(x_\alpha)& \text{if } \widehat{n}_i(x_\alpha) > 0 , \\ 
\delta_{x_i, h_i} u(x_\alpha) & \text{if } \widehat{n}_i(x_\alpha) = 0 . 
\end{cases}
\]
The discrete outward normal derivative operator ensures that $\nabla^*_h u \cdot \widehat{n}$ does not require points outside of the domain $\overline{\Om}$. 
Note that the discrete outward normal derivative approximation is only first order accurate.  
\par Next, we define the second order central difference operators for approximating second order nonmixed partial derivatives $\frac{\partial^2}{\partial x_i^2}u(x)$ by
\begin{equation}\label{cdiff2}
 \begin{array}{rcl}
\delta^2_{x_i,h_i}u({x}):= \delta^\pm_{x_i,h_i}(\delta^\mp_{x_i,h_i}(u(x)))=\frac{u({x}+h_ie_i)-2u({x})+u({x}-h_ie_i)}{h_i^2}
\end{array}
\end{equation}
for the function $u:\mathbb{R}^N \to \mathbb{R}$ and 
\begin{equation}\label{grid2}
\delta^2_{x_i,h_i}u_{h}({x}_\alpha) :=\frac{u_{h}({x}_{\alpha+e_i})-2u_{h}({x}_\alpha)+u_{h}({x}_{\alpha-e_i})}{h_i^2}
\end{equation} for all $x_{\alpha} \in \mathcal{T}_h \cap \Omega$ for the grid function $u_h:\mathcal{T}_h \to \mathbb{R}$.
Finally, we define the second order discrete Laplacian operator  $\Delta_{{h}}$ by $$ \Delta_{{h}} := \sum_{i=1}^N\delta^2_{x_i,h_i}.$$

\subsection{Formulation and Analysis for the Single Equation Case}
\label{sec:FD_formulation}
We use the following discrete problem to approximate solutions to \eqref{pde_lambda}, 
where the grid function $u_{h} : \mathcal{T}_h \to \mathbb{R}$ is an approximation for the PDE solution $u$ 
over the grid $\mathcal{T}_h$:
\begin{equation}\label{d_sol_single}
    \left\lbrace\begin{array}{rcll}
    -\Delta_{h} u_{h}+ u_{h}&=&0 & \mbox{ in } \mathcal{T}_h \cap \Omega \,, \\
    \nabla_h^* u_{h}  \cdot \widehat{n} - \lambda f(u_{h}  ) &=& 0 & \mbox{ on } \mathcal{T}_h \cap \widetilde{\partial \Omega} \,. 
    \end{array}\right.
\end{equation}
We characterize when \eqref{d_sol_single} is guaranteed to have a nonnegative solution 
while allowing for the existence of multiple solutions.  
To this end, we apply the Schauder Fixed Point Theorem to a modified version of \eqref{d_sol_single} 
where we apply cutoffs for the nonlinearity on the boundary.  
When applying the methods, the cutoffs can be chosen such that the solution to the modified problem agrees with the solution to \eqref{d_sol_single} based on an {\it a priori} bound. The admissibility analysis using sub- and supersolution techniques with cutoff operators builds upon recent developments for non-monotone finite difference methods in \cite{Lewis_Xue_2025}, although our framework must account for the specific challenges posed by superlinear boundary nonlinearities.
We will also characterize qualitative aspects of solutions to \eqref{d_sol_single} as well as its modified version.  

Choose $\rho, K > 0$ with $\rho < \frac{K}{h_*}$, 
and define the function $\widetilde{f}: [0,\infty) \to \left[ \frac{\rho}{\lambda},\frac{K}{\lambda h^*}\right]$ by 
\begin{equation}\label{f_tilde}
    \widetilde{f}(s) := \begin{cases}
        \frac{K}{\lambda h^*} & \text{if } f(s) > \frac{K}{\lambda h^*} , \\ 
        \frac{\rho}{\lambda} & \text{if } f(s) < \frac{\rho}{\lambda} , \\ 
        f(s) & \text{otherwise}. 
    \end{cases}
\end{equation}
We prove the following theorem that guarantees the admissibility of a modified version of \eqref{d_sol_single} with $f$ replaced by $\widetilde{f}$.
\begin{lem}\label{thm:admissibile}
There exists a nonnegative grid function $\widetilde{u}_h : \mathcal{T}_h \to \mathbb{R}$ such that 
\begin{equation}\label{d_sol_single_tilde}
    \left\lbrace\begin{array}{rcll}
    -\Delta_{h} \widetilde{u}_{h}+ \widetilde{u}_{h}&=&0 & \text{ in } \mathcal{T}_h \cap \Omega \,, \\
    \nabla_h^* \widetilde{u}_{h}  \cdot \widehat{n} - \lambda \widetilde{f}(\widetilde{u}_{h}  ) &=& 0 & \text{ on } \mathcal{T}_h \cap \widetilde{\partial \Omega} . 
    \end{array}\right.
\end{equation}    
\end{lem}
\begin{proof}
Let $M_K > 0$ be defined by $M_K := 2 N \frac{K}{h_*^2}$.  
Define the grid function $\overline{u}_h : \mathcal{T}_h \to \mathbb{R}$ by 
\begin{equation}
    \overline{u}_h(x_\alpha) := 
    \begin{cases}
        M_K & \text{if } x_\alpha \in \mathcal{T}_h \cap \Omega , \\ 
        M_K + K & \text{if } x_\alpha \in \mathcal{T}_h \cap \partial \Omega . 
    \end{cases}
\end{equation}
Define the subgrid $\mathring{\mathcal{T}}_h$ by 
\[
    \mathring{\mathcal{T}}_h := \left\{ x_\alpha \in \mathcal{T}_h \mid 
    x_{\alpha \pm \mathbf{e}_i} \notin \partial \Omega \text{ for all } i=1,2,\ldots,N \right\} . 
\]
Then, there holds 
\begin{align*}
    -\Delta_h \overline{u}_h(x_\alpha) + \overline{u}_h(x_\alpha) 
    & = 0 + M_K \\ 
    & = M_K \geq 0 
\end{align*}
for all $x_\alpha \in \mathring{\mathcal{T}}_h$.  
Choose $x_\alpha \in (\mathcal{T}_h \setminus \mathring{\mathcal{T}}_h) \cap \Omega$, 
and define 
\[
    E_i^\pm (x_\alpha) := 
    \begin{cases}
        1 & \text{if } x_{\alpha \pm {e}_i} \in \partial \Omega , \\ 
        0 & \text{otherwise} 
    \end{cases}
\]
for all $i = 1,2,\ldots,N$.  
Then 
\begin{align*}
    -\Delta_h \overline{u}_h(x_\alpha) + \overline{u}_h(x_\alpha)
    & = \sum_{i=1}^N E_i^+(x_\alpha) \frac{M_K - (M_K+K)}{h_i^2} + \sum_{i=1}^N E_i^-(x_\alpha) \frac{M_K - (M_K+K)}{h_i^2} 
        + M_K \\ 
    & \geq - 2N \frac{K}{h_*^2} + M_K = 0 . 
\end{align*}
Lastly, for $x_\alpha \in \mathcal{T}_h \cap \widetilde{\partial \Omega}$, there holds 
\begin{align*}
    \nabla_h^* \overline{u}_{h}  \cdot \widehat{n} - \lambda \widetilde{f}(\overline{u}_{h})
    & = \frac{K}{h_i} - \lambda \widetilde{f}(M_K+K) \\ 
    & \geq \frac{K}{h^*} - \lambda \frac{K}{\lambda h^*} = 0  
\end{align*}
for some $i \in \{1,2,\ldots,N\}$.  
Hence, $\overline{u}_h$ is a positive supersolution of \eqref{d_sol_single_tilde}.  
Next, we define the grid function $\underline{u}_h : \mathcal{T}_h \to \mathbb{R}$ by 
\begin{equation}
    \underline{u}_h(x_\alpha) := 0 
\end{equation}
for all $x_\alpha \in \mathcal{T}_h$.  
Then, 
\begin{align*}
    -\Delta_h \underline{u}_h(x_\alpha) + \underline{u}_h(x_\alpha) 
    & = 0 
\end{align*}
for all $x_\alpha \in \mathcal{T}_h \cap \Omega$ 
and 
\begin{align*}
    \nabla_h^* \underline{u}_{h}  \cdot \widehat{n} - \lambda \widetilde{f}(\underline{u}_{h})
    & = - \lambda \widetilde{f}(0) \\ 
    & \leq - \rho < 0 
\end{align*}
for all $x_\alpha \in \mathcal{T}_h \cap \partial \Omega$.  
Hence, $\underline{u}_h$ is a nonnegative subsolution of \eqref{d_sol_single_tilde} 
with $\underline{u}_h$ not an actual solution due to the cutoff $\rho$ for $\widetilde{f}$.  
Furthermore, $\underline{u}_h \leq \overline{u}_h$.  

Let $S(\mathcal{T}_h)$ denote the space of all grid functions, 
and define the nonempty convex subspace $\overline{\underline{S}}(\mathcal{T}_h)$ by 
\[
    \overline{\underline{S}}(\mathcal{T}_h) := \left\{ v_h : \mathcal{T}_h \to \mathbb{R} \mid 
    \underline{u}_h \leq v_h \leq \overline{u}_h \right\} . 
\]
We use the Schauder fixed point theorem to show that \eqref{d_sol_single_tilde} has a solution in the space 
$\overline{\underline{S}}(\mathcal{T}_h)$ from which the result follows.  
Thus, \eqref{d_sol_single_tilde} has a nonnegative solution.  
Define the mapping $\mathcal{M} : S(\mathcal{T}_h) \to S(\mathcal{T}_h)$ by 
$w_h := \mathcal{M} v_h$ if 
\begin{equation}\label{cM_mapping}
    \left\lbrace\begin{array}{rcll}
    -\Delta_{h} w_{h}+ w_{h}&=&0 & \mbox{ in } \mathcal{T}_h \cap \Omega \,, \\
    \nabla_h^* w_{h}  \cdot \widehat{n} &=& \lambda \widetilde{f}(v_{h}  ) & \mbox{ on } \mathcal{T}_h \cap \widetilde{\partial \Omega} 
    \end{array}\right.
\end{equation}
for $v_h \in S(\mathcal{T}_h)$. 
Clearly a fixed point of $\mathcal{M}$ is a solution to \eqref{d_sol_single_tilde}.  

Let $n = | \mathcal{T}_h \cap (\Omega \cup \widetilde{\partial \Omega}) |$. 
Then, there exists a matrix $A \in \mathbb{R}^{n \times n}$ and a vector $\vec{b} \in \mathbb{R}^n$ such that 
\eqref{cM_mapping} is equivalent to solving $A \vec{w} = \vec{b}$, 
where $\vec{w}$ is the vectorization of $w_h$ 
and $\vec{b}$ depends on $\vec{v}$, the vectorization of $v_h$.  
By the choice of the difference operator $\nabla_h^*$ used on the boundary 
and the form of the equations corresponding to $-\Delta_h w_h + w_h$, 
we have the matrix $A$ is a symmetric Z-matrix.  
By Gershgorin's circle theorem, all of the eigenvalues of $A$ are nonnegative.  
Since constant functions are not in the nullspace of $A$, we actually have the matrix must be 
nonsingular using properties of the corresponding discretization of Poisson's equation with Neumann boundary conditions.  
Therefore, $A$ is a nonsingular M-matrix.  
It follows that $w_h$ is uniquely defined by $v_h$ in \eqref{cM_mapping} and that the mapping $\mathcal{M}$ is well-defined.  

Suppose $v_h \in \overline{\underline{S}}(\mathcal{T}_h)$, and let $w_h = \mathcal{M} v_h$.  
We show $w_h \in \overline{\underline{S}}(\mathcal{T}_h)$ 
from which it follows that $\mathcal{M}$ has a fixed point in the space $\overline{\underline{S}}(\mathcal{T}_h)$.  
Observe that 
\begin{align*}
    -\Delta_h \underline{u}_h(x_\alpha) + \underline{u}_h(x_\alpha) 
    & = 0 
    = -\Delta_{h} w_{h}+ w_{h} 
    \leq -\Delta_h \overline{u}_h(x_\alpha) + \overline{u}_h(x_\alpha) 
\end{align*}
for all $x_\alpha \in \mathcal{T}_h \cap \Omega$ 
and 
\begin{align*}
    \nabla_h^* \underline{u}_{h}  \cdot \widehat{n} 
    & = 0 
    \leq \lambda \widetilde{f}(v_h) 
    = \nabla_h^* w_{h}  \cdot \widehat{n}  
\end{align*}
with
\begin{align*}
    \lambda \widetilde{f}(v_h) 
    & \leq \lambda \frac{K}{\lambda h^*} 
    \leq \frac{K}{h_i} 
    = \nabla_h^* \overline{u}_{h}  \cdot \widehat{n} 
\end{align*}
for some $i \in \{1,2,\ldots,N\}$ 
for all $x_\alpha \in \mathcal{T}_h \cap \widetilde{\partial \Omega}$.  
Letting $\underline{U}, \vec{w}, \overline{U} \in \mathbb{R}^n$ denote the vectorizations of 
$\underline{u}_h, w_h, \overline{u}_h$, respectively, 
it follows that 
\[
    A \underline{U} \leq A \vec{w} \leq A \overline{U}  
\]
using the natural partial ordering for vectors.  
Since $A$ is an M-matrix, there holds $\underline{U} \leq \vec{w} \leq \overline{U}$, 
and we can conclude $w_h \in \overline{\underline{S}}(\mathcal{T}_h)$.  
\end{proof}

We now explore qualitative properties of any nonnegative solution to \eqref{d_sol_single_tilde}.  
In particular, we show that no solution can be zero-valued at any node in the grid $\mathcal{T}_h \cap (\Omega \cup \widetilde{\partial \Omega})$ and any solution must have a maximum value on the boundary.  

\begin{lem}[Positivity of Solutions]
\label{lem:positivity}
Let $u_h$ be a nonnegative solution to the modified problem \eqref{d_sol_single_tilde}. Then $u_h > 0$ on $\mathcal{T}_h \cap (\Omega \cup \widetilde{\partial \Omega})$.
\end{lem}

\begin{proof}
Let $u_h$ be a nonnegative solution to \eqref{d_sol_single_tilde}. 
Suppose $u_h(x_\alpha) = 0$ for some $x_\alpha \in \mathcal{T}_h \cap (\Omega \cup \widetilde{\partial \Omega})$.  
First, suppose $x_\alpha \in \mathcal{T}_h \cap \widetilde{\partial \Omega}$.  
Then, 
$\nabla_h^* u_h(x_\alpha) \cdot \widehat{n} = \lambda \widetilde{f}(u_h(x_\alpha)) = \lambda \widetilde{f}(0) \geq \rho > 0$.  
Thus, there exists $i \in \{1,2,\ldots,N\}$ such that 
$\delta_{x_i, h_i}^+ u_h(x_\alpha) > 0$ or $\delta_{x_i, h_i}^- u_h(x_\alpha) > 0$, and it follows that 
$u_h(x_{\alpha+e_i}) < u_h(x_\alpha) = 0$ or $u_h(x_{\alpha-e_i}) < u_h(x_\alpha) = 0$, 
a contradiction.  
Thus, we must have $x_\alpha \in \mathcal{T}_h \cap \Omega$.  

Observe that $0 = u_h(x_\alpha) - \Delta_h u_h(x_\alpha) = -\Delta_h u_h(x_\alpha)$ implies 
\[
    0 = \sum_{i=1}^N \frac{2}{h_i^2} u_h(x_\alpha) = 
    \sum_{i=1}^N \frac{1}{h_i^2} \big( u_h(x_{\alpha-\mathbf{e}_i}) + u_h(x_{\alpha+\mathbf{e}_i}) \big) 
\]
with $u_h(x_{\alpha\pm\mathbf{e}_i}) \geq 0$ for all $i=1,2,\ldots,N$.  
Thus, we must have $u_h(x_{\alpha\pm\mathbf{e}_i}) = 0$ for all $i= 1,2,\ldots,N$.  
Repeating this argument, we have $u_h(x_{\alpha}) = 0$ for all 
$x_\alpha \in \mathcal{T}_h \cap (\Omega \cup \widetilde{\partial \Omega})$, 
a contradiction to the fact $u_h$ cannot be zero-valued on $\mathcal{T}_h \cap \widetilde{\partial \Omega}$.  
Therefore, we must have $u_h$ is positive over $\mathcal{T}_h \cap (\Omega \cup \widetilde{\partial \Omega})$.  
\end{proof}

\begin{lem}[Discrete Maximum Principle]
\label{lem:discrete_max_principle}
Let $u_h$ be a nonnegative solution to the modified problem \eqref{d_sol_single_tilde}. Then $u_h : \mathcal{T}_h \cap (\Omega \cup \widetilde{\partial \Omega}) \to \mathbb{R}$ achieves its maximum value on the boundary $\mathcal{T}_h \cap \widetilde{\partial \Omega}$.
\end{lem}

\begin{proof}
Let $u_h$ be a nonnegative solution to \eqref{d_sol_single_tilde}, and suppose $u_h$ is maximized at 
$x_\alpha \in \mathcal{T}_h \cap \Omega$.  
Then, $u_h(x_\alpha)$ must be positive, and it follows that 
$-\Delta_h u_h(x_\alpha) = - u_h(x_\alpha) < 0$.  
Thus, 
\[
    \sum_{i=1}^N \frac{2}{h_i^2} u_h(x_\alpha) 
    < \sum_{i=1}^N \frac{1}{h_i^2} \big( u_h(x_{\alpha-\mathbf{e}_i}) + u_h(x_{\alpha+\mathbf{e}_i}) \big) 
    \leq \sum_{i=1}^N \frac{2}{h_i^2} u_h(x_\alpha) , 
\]
a contradiction.  
Therefore, any solution must achieve its maximum value at $x_\alpha \in \mathcal{T}_h \cap \widetilde{\partial \Omega}$.  
\end{proof}

\begin{rem}\label{rem:rhozero}
Observe that results similar to Lemma~\ref{lem:positivity} and Lemma~\ref{lem:discrete_max_principle} hold for any nontrivial solution to \eqref{d_sol_single}. 
If $f(0) = 0$, then the zero-valued function is a trivial nonnegative solution to \eqref{d_sol_single}.  
Thus, Theorem~\ref{thm:admissibile} and the above lemmas can trivially be extended to the case $\rho \geq 0$ when seeking 
nonnegative solutions.  
Using $\rho > 0$ when $f(0) = 0$ eliminates the trivial solution from \eqref{d_sol_single_tilde} 
which can be useful when seeking nontrivial solutions of \eqref{d_sol_single} when they exist 
or deciding no positive solution of \eqref{d_sol_single} exists for a given $\lambda$ value when $f'(0) > 0$.  
\end{rem}

\begin{rem}\label{rem:implications}
The problem \eqref{d_sol_single_tilde} may have multiple solutions 
with one solution being the trivial solution if $f(0) = 0$ and $\rho = 0$.  
By the results in \cite{Lewis_Morris_Zhang_2022}, any convergent subsequence of solutions as $h \to 0^+$ will converge to a solution to the PDE.  
If the scheme has a stable sequence of solutions, then it is guaranteed that a convergent subsequence exists.  
Letting $K \to \infty$, we will eventually have 
$\frac{K}{h^*}$ bounds $f(u)$ and 
$M_K$ bounds $u$ for the maximal positive solution $u$ to the PDE.  
Indeed, we observe this behavior in Sections~\ref{sec:single-equation} and \ref{sec:coupled-systems} where we compute solutions to 
\eqref{d_sol_single} as well as $\| u_h \|_\infty$ vs $\lambda$ bifurcation curves.  
We also derive {\it a priori} stability bounds for any nonnegative solution of \eqref{d_sol_single} that can 
be used in \eqref{d_sol_single_tilde} as characterized in Corollary~\ref{cor:stability} below.  
Note that it is easy to check to see if a solution $u_h$ to \eqref{d_sol_single_tilde} 
is a solution to \eqref{d_sol_single} by simply checking to see if 
$f(u_h(x_\alpha)) = \widetilde{f}(u_h(x_\alpha))$ for all $x_\alpha \in \mathcal{T}_h \cap \widetilde{\partial \Omega}$. 
\end{rem}

We now use the superlinearity of the nonlinear boundary condition to 
derive an {\it a priori} bound for any solution of \eqref{d_sol_single}.  
First note that, by the superlinear and subcritical growth condition of $f$, since $p > 1$ and $\lim_{s \to +\infty} \frac{f(s)}{s^p} = b > 0$, we have $f(s) \sim bs^p$ as $s \to \infty$, which implies $\frac{f(s)}{s} \sim bs^{p-1} \to \infty$ as $s \to \infty$ since $p-1 > 0$.
Consequently, there exists a constant $C>0$ such that 
\begin{equation}\label{a_priori_bound}
    \frac{f(s)}{s} > \frac{1}{\lambda h_*} \qquad \text{for all } s > C . 
\end{equation}

\begin{thm}\label{thm:stability}
Let $C > 0$ be defined by \eqref{a_priori_bound}, and let $u_h$ be a nonnegative solution to \eqref{d_sol_single}.  
Then $\max_{x_\alpha \in \mathcal{T}_h \cap (\Omega \cup \widetilde{\partial \Omega})} u_h (x_\alpha) \leq C$.  
\end{thm}

\begin{proof}
The result clearly holds if $u_h(x_\alpha) = 0$ for all $x_\alpha \in \mathcal{T}_h \cap (\Omega \cup \widetilde{\partial \Omega})$.  Suppose 
\[
\max_{x_\alpha \in \mathcal{T}_h \cap (\Omega \cup \widetilde{\partial \Omega})} u_h (x_\alpha) > 0.
\]
Similar to the proof of Lemma~\ref{lem:discrete_max_principle}, $u_h$ must achieve its maximum over $\mathcal{T}_h \cap \widetilde{\partial \Omega}$.  Let $x_\alpha \in \mathcal{T}_h \cap \widetilde{\partial \Omega}$ be the point at which $u_h$ is maximized.  
Then, by the nonlinear boundary condition, 
$\nabla_h^* u_h(x_\alpha) \cdot \widehat{n} = \lambda f( u_h(x_\alpha))$, 
and it follows that there exists an index $i$ such that 
\begin{align*}
- \delta_{x_i, h_i}^+ u_h(x_\alpha) = \lambda  f( u_h(x_\alpha)) \qquad \text{or} \qquad 
\delta_{x_i, h_i}^- u_h(x_\alpha) = \lambda  f( u_h(x_\alpha)) . 
\end{align*}
Observe that, since $u_h$ is nonnegative, 
\begin{align*}
- \delta_{x_i, h_i}^+ u_h(x_\alpha) = \lambda  f( u_h(x_\alpha)) 
& \implies 
f( u_h(x_\alpha)) \leq \frac{1}{\lambda h_i} u_h(x_\alpha) , \\ 
\delta_{x_i, h_i}^- u_h(x_\alpha) = \lambda  f( u_h(x_\alpha)) 
& \implies 
f( u_h(x_\alpha)) \leq \frac{1}{\lambda h_i} u_h(x_\alpha) . 
\end{align*}
Therefore, 
\[
\frac{f( u_h(x_\alpha))}{u_h(x_\alpha)} \leq \frac{1}{\lambda h_*} , 
\]
a contradiction unless $u_h(x_\alpha) \leq C$.  
\end{proof}

\begin{cor}\label{cor:stability}
If a nonnegative solution to \eqref{d_sol_single} exists, then there must be a nonnegative solution $\widetilde{u}_h$ to \eqref{d_sol_single_tilde} 
with $\rho = 0$, $K \geq \frac{C h^*}{h_*}$, and $M_K \geq C$ 
for which $\widetilde{u}_h$ is also a solution to \eqref{d_sol_single}, 
where $C$ is defined by \eqref{a_priori_bound}.  
\end{cor}

\begin{rem}
\label{rem:limitations}
While Corollary~\ref{cor:stability} establishes a way to search for solutions to \eqref{d_sol_single}, it does not guarantee when a nonnegative solution exists.  Instead, it allows the use of \eqref{d_sol_single_tilde} to search for a solution to \eqref{d_sol_single} where the nonlinearity for \eqref{d_sol_single_tilde} does not have superlinear growth.  The gap is natural in the sense that \eqref{d_sol_single} may not have a positive solution whereas \eqref{d_sol_single_tilde} targets positive solutions.  The {\it a priori} bound also does not ensure numerical stability as $h^* \to 0^+$ despite the fact it is observed experimentally in Sections~\ref{sec:single-equation} and \ref{sec:coupled-systems} when computing bifurcation diagrams.  Assuming numerical stability, the convergence analysis technique used in \cite{Lewis_Morris_Zhang_2022} can be applied to \eqref{d_sol_single} to guarantee the accuracy of the approximations.
The {\it a priori} bound is consistent with the fact positive solutions to the PDE blow-up as $\lambda \to 0^+$.  
\end{rem}

\section{Applications to the Single Equation Case} 
\label{sec:single-equation}

We implement the FD method \eqref{d_sol_single} for several one-dimensional test problems to validate the utility of the method and expand its use for generating approximate bifurcation curves for problem \eqref{pde_lambda}.
From the theoretical results in \cite{BCDMP_2024}, there is a bifurcation of positive solutions from infinity at $\lambda=0$. 
We will also perform an eigenvalue analysis for the one-dimensional problem to help generate an initial guess for directly solving \eqref{d_sol_single} from which we use continuation to find the complete bifurcation curve.
The eigenvalue analysis will further help us benchmark the performance of our methods for capturing when positive solutions bifurcate from the trivial solution when $f(0) = 0$ and $f'(0)>0$.  
Alternatively, we were able to start with a large constant value for the initial guess when solving for a small $\lambda$ value exploiting the fact that positive solutions bifurcate from infinity when $\lambda \to 0^+$.  
The method will be implemented in MATLAB and use the \texttt{fsolve} command to solve the corresponding nonlinear system of algebraic equations.  
We also make sure the minimum of the solution $u_h$ is positive so that we can guarantee \texttt{fsolve} only finds positive solutions when seeking nontrivial solutions. Once we find a solution we plot $\max\{|u_h|\}$ vs $\lambda$ to approximate the  $\|u\|_{\infty}$ vs $\lambda$ bifurcation curve for \eqref{pde_lambda}.

\subsection{Principal Eigenvalue for Single Equation Case}
We first perform a principal eigenvalue analysis for the single equation case when $N=1$ with $\Omega = (0,1)$.  
The analysis will identify the exact value $\lambda_1$ where the $\| u \|_{\infty}$ vs $\lambda$ bifurcation diagram bifurcates  from the trivial solution when $f(0) = 0$ and $f'(0) > 0$.  
Since $\lambda_1 = \frac{\mu_1}{f'(0)}$, where $\mu_1$ denotes the first Steklov eigenvalue, 
we seek to find the principal eigenvalue of the linearized one-dimensional single-equation problem:
\begin{align*}
    \begin{cases}
        ~~~~ -\phi'' + \phi = 0 ~;~ (0,1), \\
        ~~ -\phi'(0) - \lambda f'(0) \phi(0) = 0, \\
        ~~~~ \phi'(1) - \lambda f'(0) \phi(1) = 0.
    \end{cases}
\end{align*}

The general solution to the differential equation has the form:
\begin{align*}
    \phi(x) = A \cosh(x) + B \sinh(x),
\end{align*}
where \(A\) and \(B\) are constants which depend on the boundary conditions. At $x = 0$ we have $\phi(0) = A$ and $\phi'(0) = B$.
It follows that
\begin{align*}
    \phi'(x) &= A \sinh(x) + B \cosh(x).
\end{align*}
Substituting into the first boundary condition:
\begin{align*}
-\phi'(0) - \lambda f'(0) \phi(0)
    &= -B - \lambda f'(0) A 
    = 0 \implies B = -\lambda f'(0) A.
\end{align*}
It follows that
\begin{align*}
    \phi(x) &= A \cosh(x) - 
             \lambda f'(0) A \sinh(x) 
            = A \Big( \cosh(x) - \lambda
              f'(0) \sinh(x) \Big), \\
    \phi'(x) &= A \sinh(x) -
                \lambda f'(0) A \cosh(x)
             = A \Big( \sinh(x) - \lambda 
               f'(0) \cosh(x) \Big).
\end{align*}    
Observe that 
\begin{align*}
\phi(1) 
    &= A \Big( \cosh(1) - \lambda f'(0) 
       \sinh(1) \Big), \\
\phi'(1) 
    &= A \Big( \sinh(1) - \lambda f'(0) \cosh(1) \Big).
\end{align*}

Substituting into the second boundary condition:
\begin{align} \label{solve_for_lambda1}
\phi'(1) - \lambda f'(0) \phi(1)
    &= A \Big( \sinh(1) - \lambda f'(0) \cosh(1) \Big) - \lambda f'(0) \notag 
    A \Big( \cosh(1) - \lambda f'(0) 
       \sinh(1) \Big) \notag \\
    &= A \Bigg( [f'(0)]^2 \sinh(1) \lambda^2  - 2 f'(0) \cosh(1) \lambda 
    + \sinh(1) \Bigg) = 0.
\end{align}
Assuming \(A \neq 0\), we apply the quadratic formula to obtain:
\begin{align*}
\lambda 
    &= \frac{2 f'(0) \cosh(1) \pm \sqrt{[2 f'(0) \cosh(1)]^2 - 4 \cdot [f'(0)]^2 \sinh(1) \cdot \sinh(1)}}{2 [f'(0)]^2 \sinh(1)} \\[2ex]
    &= \frac{2 f'(0) \cosh(1) \pm 
       \sqrt{4 [f'(0)]^2 \cosh^2(1) -
       4 [f'(0)]^2 \sinh^2(1)}}
       {2 [f'(0)]^2 \sinh(1)} \\[2ex]
    &= \frac{2 f'(0) \cosh(1) \pm 
       \sqrt{4 [f'(0)]^2 \Big( \cosh^2(1) -
       \sinh^2(1) \Big)}}
       {2 [f'(0)]^2 \sinh(1)} \\[2ex]
    &= \frac{2 f'(0) \cosh(1) \pm 
       2 f'(0) \sqrt{1}}
       {2 [f'(0)]^2 \sinh(1)} \\[2ex]
    &= \frac{\cosh(1) \pm 1}{f'(0) \sinh(1)}.
\end{align*}

The principal eigenvalue is the smallest \emph{positive} eigenvalue. In all of our choices for $f$, we have $f'(0) > 0$. Furthermore, $\cosh(1) > 1$. Hence, 
\begin{align} \label{lambda1_single}
    \lambda_1 &= \displaystyle \frac{\cosh(1)-1}{f'(0) \sinh(1)}.
\end{align}

Let $\phi_1$ denote the eigenfunction corresponding to $\lambda_1$. We will take $A = 1$ (for convenience) so that $||\phi_1||_{\infty} = 1$. Here we justify that $\phi_1(x) > 0$ on $(0,1)$. Observe that
\begin{align} \label{single_eigenfunction}
\phi_1(x)
    &= \cosh(x) - \lambda_1 f'(0) \sinh(x) \notag \\
    &= \cosh(x) - \Bigg( \frac{\cosh(1)-1}{f'(0) \sinh(1)} \Bigg) f'(0) \sinh(x) \notag \\
    &= \cosh(x) - \Bigg( \frac{\cosh(1)-1}{\sinh(1)} \Bigg) \sinh(x).
\end{align}
Now,
\begin{align*}
\phi_1(x) > 0
    &\iff \cosh(x) - \Bigg( \frac{\cosh(1)-1}{\sinh(1)} \Bigg) \sinh(x) > 0
    \iff \coth(x) > \frac{\cosh(1)-1}{\sinh(1)}.
\end{align*}
Recall that $\coth$ is a decreasing function, so we easily see that for $x \in (0,1]$, we have $\coth(x) 
    \geq \coth(1) 
    = \displaystyle \frac{\cosh(1)}{\sinh(1)} > \frac{\cosh(1)-1}{\sinh(1)}.$

\subsection{Coding Algorithm for One-Dimensional Bifurcation Diagrams}

First we divide our interval $\Omega = (0,1)$ into $M-1$ subintervals with uniform width $h>0$ with $x_1 = 0$ and $x_M = 1$. We approximate  $u(x_i)$ by $u_i$ using the finite difference method to derive a nonlinear algebraic system $F(\vec{u})=\vec{0}$ to define $\vec{u}$. By \eqref{d_sol_single}, the function $F: \mathbb{R}^M \to \mathbb{R}^M$ is defined by $F_1(\vec{u})=\frac{u_1-u_2}{h}-\lambda f (u_1)$, $F_i(\vec{u})=-\frac{1}{h^2}u_{i-1}+(\frac{2}{h^2}+1)u_i-\frac{1}{h^2}u_{i+1}$ for $i \in \{2,3,\ldots , M-1\}$, and $F_M(\vec{u})=\frac{u_M-u_{M-1}}{h}-\lambda f(u_M)$. After we define the function $F$, we use \texttt{fsolve} in MATLAB to solve $F(\vec{u})=\vec{0}$ via the command:
$$u=\text{fsolve}(F,u_0,\text{flags})$$
where $u_0$ is the initial guess and flags sets our tolerances. Once we find a solution we plot $\max\{|\vec{u}|\}$ vs $\lambda$ to approximate the  $\|u\|_{\infty}$ vs $\lambda$ bifurcation curve for \eqref{d_sol_single}.

In the tests, we build the bifurcation curves using continuation on a grid of width $\Delta \lambda = 0.001$ for finer resolution near bifurcation points.  
We use a tolerance of $10^{-6}$ for both the residual tolerance and step tolerance in \texttt{fsolve}.  
A solution at $\lambda$ is used as an initial guess when seeking a solution for $\lambda' = \lambda\pm \Delta \lambda$ 
depending on the direction for constructing the bifurcation curve.  
The initial guess when first solving $F(\vec{u}) = \vec{0}$ for a given $\lambda$ value is taken to be of the form 
$\phi_1(x) = A[\cosh(x) - \frac{\cosh(1)-1}{\sinh(1)} \sinh(x)]$ for a chosen constant $A$ based on the 
principal eigenvalue analysis for the underlying PDE. 
We also ensure any positive solution is bounded below by $10^{-12}$ when seeking to approximate the bifurcation point 
and when determining if a vector corresponds to a valid positive solution of \eqref{d_sol_single}.

\subsection{Bifurcation Diagrams and Approximate Solutions}
We plot solutions of \eqref{d_sol_single} with various $f$ and $\lambda$ values. We see the graphs are concave up with the maximum values along the boundary. We also see that the maximum of the solution decreases as $\lambda$ increases when considering solutions corresponding to the branch of the bifurcation curves bifurcating from $\infty$ at $\lambda = 0$. We use the eigenfunction from \eqref{single_eigenfunction}, namely $\phi_1(x) = A[\cosh(x) - \frac{\cosh(1)-1}{\sinh(1)} \sinh(x)]$ (with $A = 1$ chosen for convenience), for the initial guess for a given initial $\lambda$ value before using continuation to trace the remaining points on the bifurcation curve.
The $\lambda$ value for the initial instance of the problem is chosen to be near but less than the bifurcation point when such a point exists.  
The results for $f(s) = s^2$ can be found in Figure~\ref{fig:single_fprime_0}, 
$f(s) = 2s+s^2$ can be found in Figure~\ref{fig:single_fprime_1}, 
and $f(s) = 0.1s - 0.1s^2 + s^3$ can be found in Figure~\ref{fig:single_cubic}.


\begin{figure}[H]
\begin{minipage}{.45\textwidth}
    \includegraphics[scale=0.53]{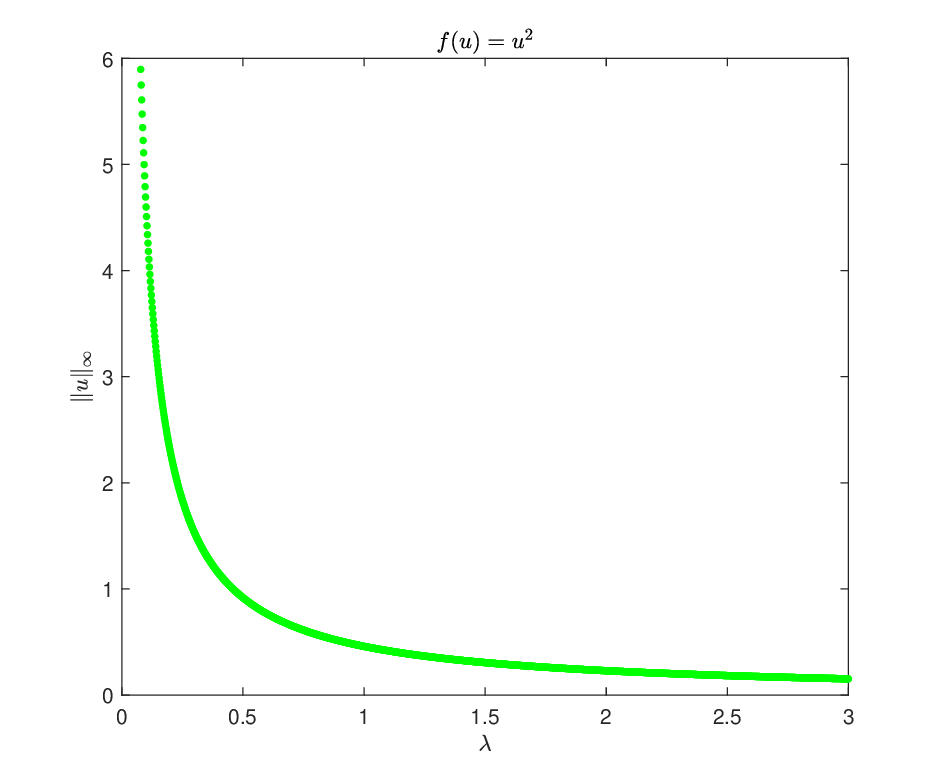}
\end{minipage}
\hspace{.05in}
\begin{minipage}{.45\textwidth}
    \includegraphics[scale=0.50]{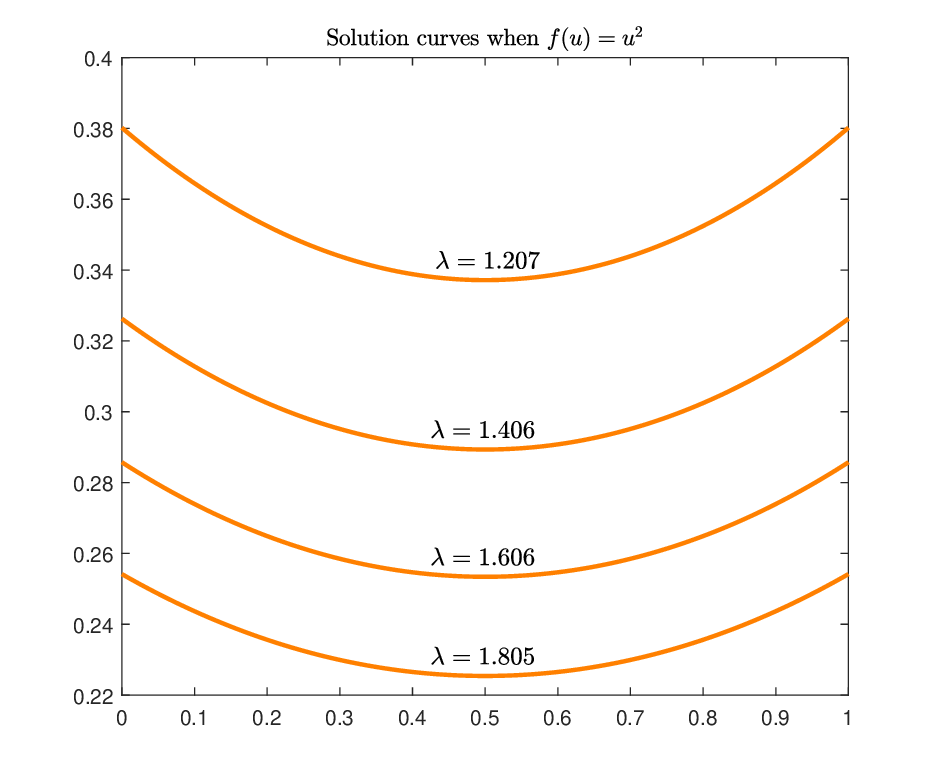}
\end{minipage}
\caption{Bifurcation diagram and sample positive solutions for $f(s) = s^2$ with $f'(0) = 0$ using $M=151$. Here $\lambda_1 = \infty$ (see \eqref{lambda1_single}). We begin by fixing an arbitrary value (not too large) $\lambda_0 = 3$ and sweep left along $[0.01,\lambda_0]$. This validates Theorem 1.1 from \cite{BCDMP_2024} showing bifurcation from infinity only.}
\label{fig:single_fprime_0}
\end{figure}



\begin{figure}[H]
\begin{minipage}{.45\textwidth}
    \includegraphics[scale=0.53]{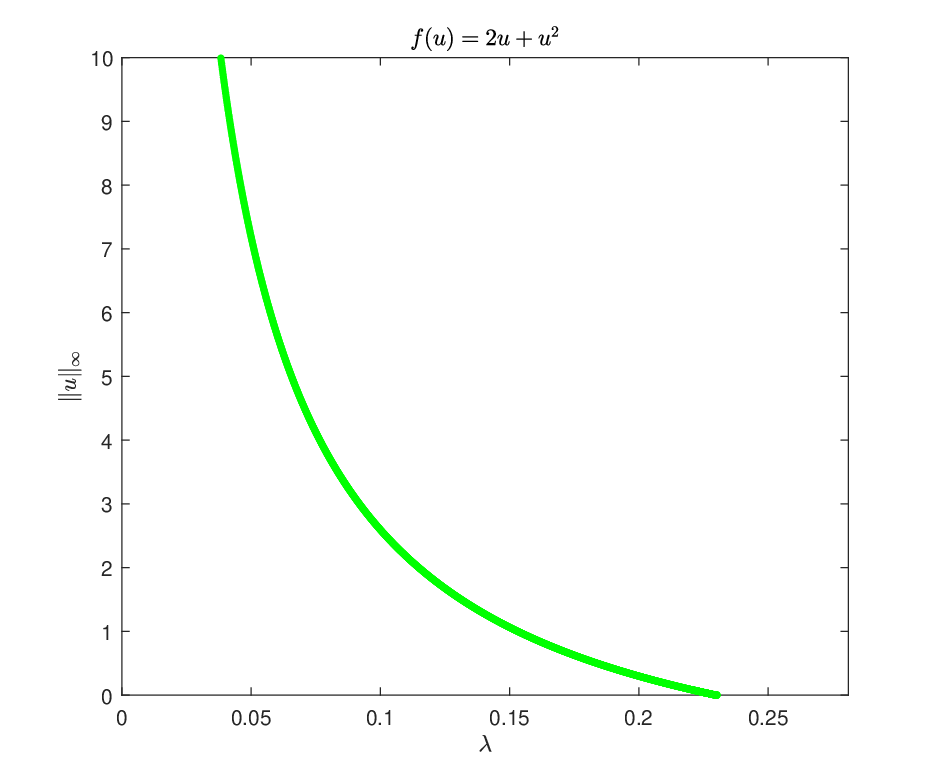}
\end{minipage}
\hspace{.05in}
\begin{minipage}{.45\textwidth}
    \includegraphics[scale=0.51]{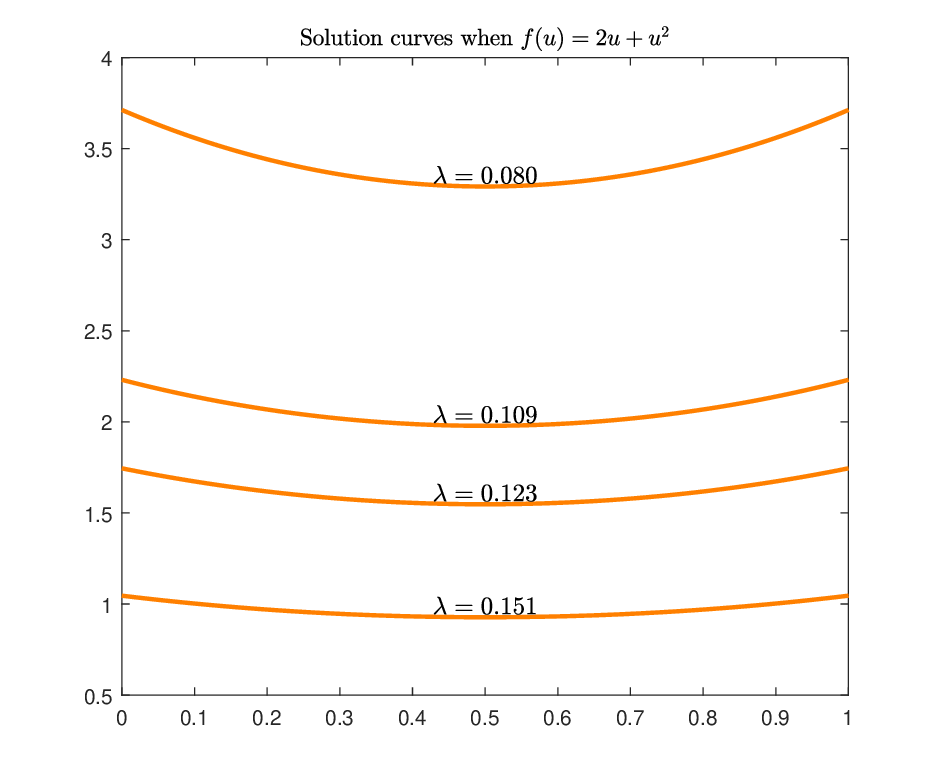}
\end{minipage}
\caption{Bifurcation diagram and sample positive solutions for $f(s) = 2s + s^2$ with $f'(0) = 2$ using $M=175$. Using \eqref{lambda1_single}, $\lambda_1 \approx 0.23105858$ is the location of the bifurcation point. We begin by sweeping left along $[0.01,\lambda_1-\delta]$ with $\delta > 0$ and small. We then sweep right along $[\lambda_1-\delta, \lambda_1]$, initially using the stored solution at $\lambda_1-\delta$. This shows subcritical bifurcation as predicted by theory.}
\label{fig:single_fprime_1}
\end{figure}



\begin{figure}[H]
\begin{minipage}{.45\textwidth}
    \includegraphics[scale=0.53]{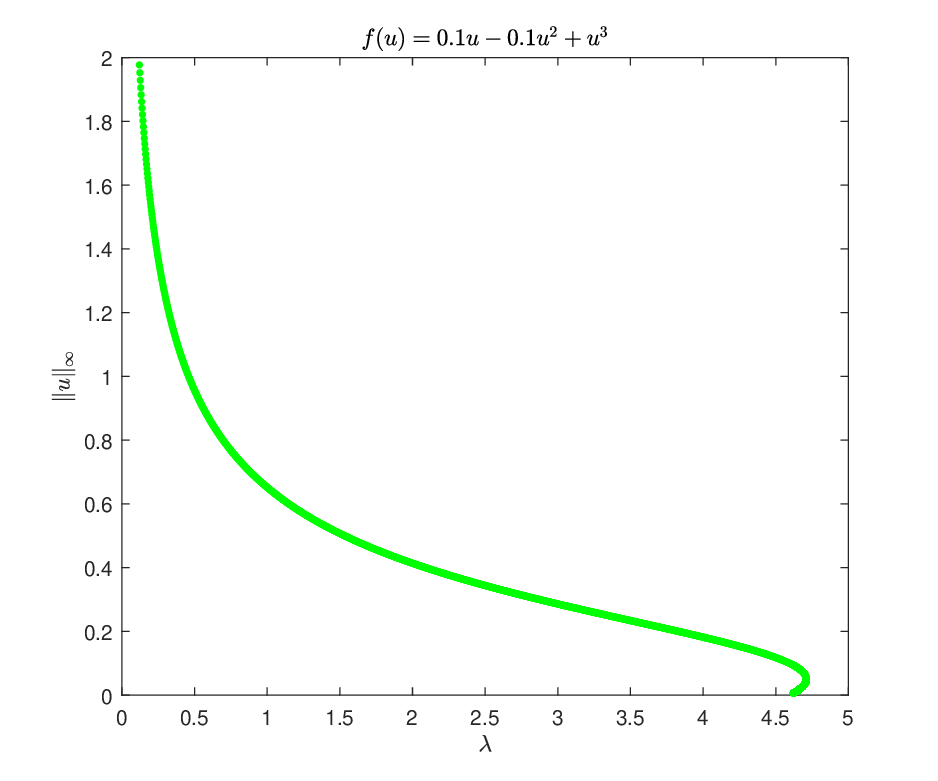}
\end{minipage}
\hspace{.05in}
\begin{minipage}{.45\textwidth}
    \includegraphics[scale=0.51]
    {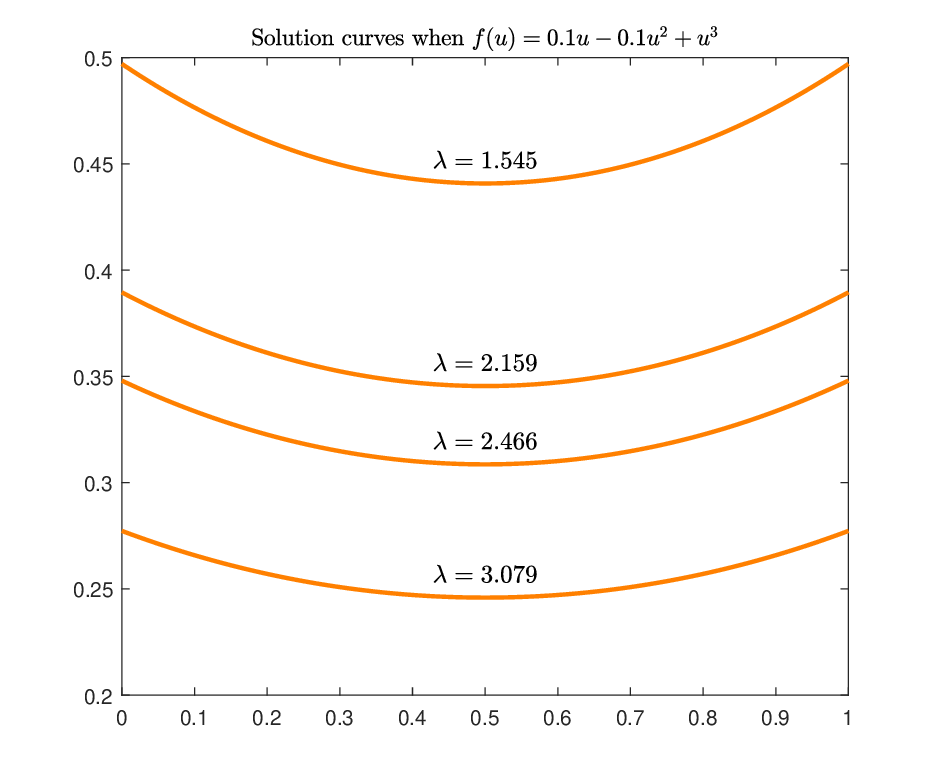}
\end{minipage}
\caption{Bifurcation diagram and sample positive solutions for $f(s) = 0.1s - 0.1s^2 + s^3$ with $f'(0) = 0.1$ using $M=101$. Using \eqref{lambda1_single}, $\lambda_1 \approx 4.62117157$ is the location of the bifurcation point. We begin by sweeping left along $[0.01,\lambda_1-\delta]$ with $\delta > 0$ and small. We then sweep right along $[\lambda_1-\delta, \lambda_{*}]$, initially using the stored solution at $\lambda_1-\delta$, until \emph{fsolve} fails to converge at $\lambda_{*}$ (approximately the turning point). In this case, the value just prior to $\lambda_{*}$ is $\lambda_{L} \approx 4.710532143928570$. Finally, we sweep left again along $[\lambda_1, \lambda_{L}]$ using the last converged solution at $\lambda_{L}$. Note the supercritical bifurcation and turning point behavior, confirming Theorem 1.2 multiplicity results.}
\label{fig:single_cubic}
\end{figure}


\subsection{Validation of Analytical Results from \cite{BCDMP_2024}}
\label{subsec:connection-single}
Our numerical results validate the theoretical framework established in \cite{BCDMP_2024} while also accurately approximating the principal eigenvalue $\lambda_1$ when $f(0) = 0$ and $f'(0) > 0$. Theorem 1.1 in \cite{BCDMP_2024} establishes that under hypothesis $(H_\infty)$, there exists $\hat{\lambda} > 0$ such that for all $\lambda \in (0, \hat{\lambda}]$, problem \eqref{pde_lambda} has a positive weak solution with $\|u\|_{C(\Omega)} \to \infty$ as $\lambda \to 0^+$, which is confirmed across all test cases (Figures \ref{fig:single_fprime_0}--\ref{fig:single_cubic}). When $f'(0) = 0$ (Figure \ref{fig:single_fprime_0}), bifurcation from the trivial solution occurs only at infinity, directly validating Theorem 1.1. Theorem 1.2 provides conditions for the existence of a connected component $C^+$ of positive solutions emanating from the trivial solution at $(\frac{\mu_1}{f'(0)}, 0)$ and possessing a unique bifurcation point from infinity at $\lambda = 0$. When $f'(0) > 0$ (Figures \ref{fig:single_fprime_1} and \ref{fig:single_cubic}), our computations confirm bifurcation from the trivial solution at $\lambda_1$ values matching theoretical predictions, with connected solution branches linking trivial and infinity bifurcations, and subcritical bifurcation behavior when higher-order terms are positive. Figure \ref{fig:single_cubic} additionally demonstrates supercritical bifurcation with a turning point, illustrating how higher-order terms determine bifurcation direction and create solution multiplicity, thereby validating both Theorems 1.1 and 1.2 under different parameter regimes.

\section{Application to the Coupled Systems Case} 
\label{sec:coupled-systems}

In this section we extend and test the FD method \eqref{d_sol_single} for solving \eqref{pde_lambda} to approximate 
solutions to the system of equations \eqref{PDE_Sys}.  In particular, we consider the discrete formulation 
\begin{equation}\label{d_sol}
    \left\lbrace\begin{array}{rcll}
    -\Delta_{h} u_{1,h}+ u_{1,h}&=&0 & \mbox{ in } \mathcal{T}_h \cap \Omega, \\
    -\Delta_h u_{2,h}+ u_{2,h} &=&0 & \mbox{ in } \mathcal{T}_h \cap \Omega,\\
    \nabla_h^* u_{1,h}  \cdot \widehat{n} - \lambda f( u_{2,h}  ) &=& 0 & \mbox{ on } \mathcal{T}_h \cap \widetilde{\partial \Omega},\\
   \nabla_h^* u_{2,h}  \cdot \widehat{n} - \lambda g(u_{1,h} ) &=& 0 & \mbox{ on } \mathcal{T}_h \cap \widetilde{\partial \Omega}\\
    \end{array}\right.
\end{equation}
in order to validate and gain intuition for the theoretical results obtained in \cite{BCDMP_Systems}.  
The numerical results in Section~\ref{sec:FD} can naturally be extended to \eqref{d_sol} with 
a similar finite difference method for sublinear problems analyzed in \cite{BLM_2025}.  
We again perform a principal eigenvalue analysis for the one-dimensional problem to benchmark the performance of our method when approximating the bifurcation from the trivial solution when such a value for $\lambda$ exists.  
The method is again implemented in MATLAB using \texttt{fsolve} to solve the corresponding system of nonlinear equations.  

The theoretical framework for systems is significantly richer than the single equation case, involving matrix analysis and more complex eigenvalue problems. The systems case algorithm for solving the discrete problem strategically leverages the single equation framework through an array concatenation approach. Rather than developing an entirely new discretization scheme, the coupled system is solved by extending the solution vector from $\vec{u} \in \mathbb{R}^M$ to $\vec{w} \in \mathbb{R}^{2M}$, where the first $M$ components represent the discretized $u$ values and the last $M$ components represent the discretized $v$ values. This allows the finite difference operators and interior point discretizations to be reused identically for both equations through systematic indexing. The coupling between equations manifests only in the boundary conditions, where $u$'s normal derivative depends on $v$ through $\lambda f(v)$ and $v$'s normal derivative depends on $u$ through $\lambda g(u)$. This design enables the same \texttt{fsolve} infrastructure to handle both single equations and coupled systems while maintaining code modularity and avoiding duplication of the core discretization logic.


\subsection{Principal Eigenvalue for the 1D System Case}
We first perform a principal eigenvalue analysis for the system of equations \eqref{PDE_Sys} when $N=1$ with $\Omega = (0,1)$. 
Positive solutions bifurcate from infinity at $\lambda = 0$ and, when applicable, from the trivial solution at 
\begin{align} \label{system_lambda1}
    \lambda_1 &= \frac{\mu_1}{\sqrt{f'(0)g'(0)}},
\end{align} 
where $\mu_1$ denotes the first Steklov eigenvalue. 
To benchmark our methods, we first determine the bifurcation point from the trivial solution, which corresponds to finding the principal eigenvalue $\lambda_1$ of the linearized one-dimensional system:
\begin{align*}
    \begin{cases}
        -\phi'' + \phi = 0 ~;~ x \in (0,1) , \\
        -\psi'' + \psi = 0 ~;~ x \in (0,1) , \\
        -\phi'(0) - \lambda f'(0) \psi(0) = 0 , & \\
        \hspace{.13in} \phi'(1) - \lambda f'(0) \psi(1) = 0 , &  \\
        -\psi'(0) - \lambda g'(0) \phi(0) = 0 , & \\
        \hspace{.13in} \psi'(1) - \lambda g'(0) \phi(1) = 0. & 
    \end{cases}
\end{align*}

The general solution to the differential equations \(-\phi'' + \phi = 0\) and \(-\psi'' + \psi = 0\) are:
\begin{align*}
    \phi(x) &= A \cosh(x) + B \sinh(x), \\
    \psi(x) &= C \cosh(x) + D \sinh(x),
\end{align*}
where \(A\), \(B\), \(C\), and \(D\) are constants determined by the boundary conditions. From this, it follows that
\begin{align*}
    \phi'(x) &= A \sinh(x) + B \cosh(x) , \\
    \psi'(x) &= C \sinh(x) + D \cosh(x).
\end{align*}
Using the boundary conditions at \(x = 0\), we have:
\begin{align*}
    -\phi'(0) - \lambda f'(0) \psi(0) = 0
        &\implies -B - \lambda f'(0) C = 0 
        \implies B = -\lambda f'(0) C. \\
    -\psi'(0) - \lambda g'(0) \phi(0) = 0 
        &\implies -D - \lambda g'(0) A = 0 
        \implies D = -\lambda g'(0) A.
\end{align*}
Thus, we can simplify the expression for  $(\phi, \psi)$ in terms of $A$ and $C$ only:
\begin{subequations}\label{sys_eigenfunctionsAC}
\begin{align}
    \phi(x) &= A \cosh(x) - C \lambda f'(0) \sinh(x) \\
    \psi(x) &= C \cosh(x) - A \lambda g'(0) \sinh(x).
\end{align}
\end{subequations}


Now we apply the boundary conditions at $x = 1$. Observe that
   \begin{align} \label{system_ACsolving1}
       \phi'(1) - \lambda f'(0) \psi(1) 
       &= \Big( A \sinh(1) - \lambda f'(0) C
           \cosh(1) \Big) - 
           \lambda f'(0) \Big( C \cosh(1) - \lambda g'(0) A \sinh(1) \Big) \notag \\
       &= A \sinh(1) - \lambda f'(0) C \cosh(1) - 
          \lambda f'(0) C \cosh(1) + \lambda^2 f'(0) g'(0) A \sinh(1) \notag \\
       &= A \sinh(1) - 2\lambda f'(0) C \cosh(1) + \lambda^2 f'(0) g'(0) A \sinh(1) \notag \\
       &= A \sinh(1) \Big( 1 + \lambda^2 f'(0) g'(0) \Big) - 
       2\lambda f'(0)\cosh(1) C = 0.
   \end{align}
   From this, we obtain $C$ in terms of $A$:
   \begin{align} \label{system_eigenfunction_AC}
       C &= \frac{A \sinh(1) \Big( 1 + \lambda^2 f'(0) g'(0) \Big)}{2 \lambda f'(0) \cosh(1)}
         = \Bigg[ \frac{\Big( 1 + \lambda^2 f'(0) g'(0) \Big) \tanh(1)}{2 \lambda f'(0)} \Bigg] A.
   \end{align}
Now the forms of $\phi$ and $\psi$ simplify further:
    \begin{align}
        \phi(x) 
            &=  A \Bigg( \cosh(x) - \frac{1}{2} \Big(1+ \lambda^2 f'(0) g'(0) \Big) \tanh(1) \sinh(x) \Bigg), \label{system_phi} \\
        \psi(x) 
            &= A \Bigg( \Big( \frac{1+\lambda^2 f'(0)g'(0)}{2\lambda f'(0)} \Big) \tanh(1) \cosh(x) - \lambda g'(0) \sinh(x) \Bigg). \label{system_psi}
    \end{align}


Applying the other boundary condition at $x = 1$, we have:
\begin{align} \label{system_ACsolving2}
\psi'(1) - \lambda g'(0) \phi(1)
       &= C \sinh(1) - \lambda g'(0) A \cosh(1) - 
          \lambda g'(0) \Big(A \cosh(1) - \lambda f'(0) C \sinh(1) \Big) \notag \\
       &= C \sinh(1) - \lambda g'(0) A \cosh(1) - 
          \lambda g'(0) A \cosh(1) + \lambda^2 f'(0) g'(0) C \sinh(1) \notag \\
       &= C \sinh(1) \Big( 1 + \lambda^2 f'(0) 
          g'(0) \Big) - 2\lambda g'(0) \cosh(1) A = 0.
   \end{align}
Solving for $A$, we have:
   \begin{align*}
      A &= \frac{\sinh(1) \Big( 1 + \lambda^2 
            f'(0) g'(0) \Big)}{2\lambda g'(0) \cosh(1)} \cdot C \\
        &= \frac{\tanh(1) \Big( 1 + \lambda^2 
            f'(0) g'(0) \Big)}{2\lambda g'(0)} \cdot \frac{\Big( 1 + \lambda^2 f'(0) g'(0) \Big) \tanh(1)}{2 \lambda f'(0)} \cdot A  \\
        &= \frac{\Big( 1 + \lambda^2 f'(0) 
            g'(0) \Big)^2 [\tanh(1)]^2}
            {4 \lambda^2 f'(0) g'(0)} A.
   \end{align*}
   This implies that
   \begin{align*}
       A \left( 1 - \frac{\Big( 1 + \lambda^2 f'(0) g'(0) \Big)^2 [\tanh(1)]^2}
       {4 \lambda^2 f'(0) g'(0)} \right) = 0.
   \end{align*}
   Since we assume $A \neq 0$ (to avoid the zero eigenfunction), the principal eigenvalue  $\lambda_1$ can be approximated (numerically) as the smallest positive solution of:
   \begin{align} \label{system_mu1}
       \left[ f'(0) g'(0) \right]^2 \lambda^4 
        + \left( 2 - \frac{4}{\tanh^2(1)} \right) f'(0) g'(0) \lambda^2 
        + 1 = 0.
   \end{align}

\underline{Existence of $\lambda_1 > 0$:} \\

The existence of at least one positive solution is justified by the Intermediate Value Theorem. Define
\begin{align*}
    h(\lambda) := \left[ f'(0) g'(0) \right]^2 \lambda^4 
        + \left( 2 - \frac{4}{\tanh^2(1)} \right) f'(0) g'(0) \lambda^2 
        + 1.
\end{align*}
Observe that $h(0) = 1$ and, since $h$ is a polynomial of degree 4 in $\lambda$, $h$ is continuous (and differentiable). Hence
\begin{align*}
    h'(\lambda) &= 4 [f'(0)g'(0)]^2 \lambda^3 +
                   2 \left( 2 - \frac{4}{\tanh^2(1)} \right) f'(0) g'(0) \lambda  \\
                &= 2\lambda \left( 2[f'(0)g'(0)]^2 
                   \lambda^2 + \left( 2 - \frac{4}{\tanh^2(1)} \right) f'(0) g'(0) \right).
\end{align*}
Setting $h'(\lambda) = 0$ to identify the location of a relative minimum, since $\lambda > 0$, we must have 
\begin{align*}
    2[f'(0)g'(0)]^2 \lambda^2 + \left( 2 - \frac{4}{\tanh^2(1)} \right) f'(0) g'(0) = 0.
\end{align*}
Solving for $\lambda$, we have:
\begin{align*}
    \lambda^2 &= \frac{\Big(\frac{4}{\tanh^2(1)}-2 
                 \Big) f'(0) g'(0)}{2 [f'(0)g'(0)]^2}
              = \frac{\Big(\frac{2}{\tanh^2(1)}-1 
                 \Big)}{f'(0)g'(0)}
            \implies 
    \lambda = +\sqrt{\frac{\Big(\frac{2}{\tanh^2(1)}-1 
                 \Big)}{f'(0)g'(0)}}.
\end{align*}

The above calculation indicates that there is precisely one critical point for $h$ for $\lambda \in (0,\infty)$. We claim that $h$ is decreasing on the interval $I := \left( 0, \sqrt{\frac{\Big(\frac{2}{\tanh^2(1)}-1 \Big)}{f'(0)g'(0)}} \right)$. Since $\tanh^2(x) < 1$ for all $x \in \mathbb{R}$, it follows that $\tanh^2(1) < 1$. Hence $1 < \displaystyle \frac{2}{\tanh^2(1)}-1$, which means a suitable test value to verify $h'(\lambda) < 0$ on $I$ is $\lambda_{*} := \displaystyle \frac{1}{\sqrt{f'(0)g'(0)}}$. Observe that
\begin{align*}
h(\lambda_{*}) 
        &= 2[f'(0)g'(0)]^2 \cdot \frac{1} 
           {[f'(0)g'(0)]^2} + \left( 2 - \frac{4}{\tanh^2(1)} \right) f'(0) g'(0) 
           \cdot \frac{1}{f'(0)g'(0)} + 1 \\
        &= 4 - \frac{4}{\tanh^2(1)} \\
        &= 4 \left( 1 - \coth^2(1) \right)
        < 0
\end{align*}
since $\cosh^2(1) - \sinh^2(1) = 1 > 0$. Since $h$ is continuous, there must exist a positive root of $h$ in this interval ($h$ changes from positive to negative).
In particular, we have verified the principle eigenvalue for the linearized one-dimensional system $\lambda_1$ exists, is positive, and satisfies  $\lambda_1 < \sqrt{\frac{\Big(\frac{2}{\tanh^2(1)}-1 \Big)}{f'(0)g'(0)}}$. \\


\underline{Positivity of $\phi$ on $(0,1)$:} \\

First we examine $\phi$ from \eqref{system_phi}. From our previous analysis, we know $\lambda_1 < \sqrt{\frac{\frac{2}{\tanh^2(1)}-1}{f'(0)g'(0)}}$. Therefore, $1 + \lambda_1^2 f'(0) g'(0) < \frac{2}{\tanh^2(1)}$. Hence 
\begin{align*} 
    \cosh(x) - \frac{1}{2} \Big( 1 + \lambda_1^2 f'(0) g'(0) \Big) \tanh(1) \sinh(x) 
        &> \cosh(x) - \frac{1}{2} \frac{2}{\tanh^2(1)} \tanh(1) \sinh(x) \\
        &= \cosh(x) - \coth(1) \sinh(x) \\
        &> 0,
\end{align*}
provided that $\coth(x) > \coth(1)$. Since $\coth$ is a decreasing function, we immediately see this is true for $x \in (0, 1)$. \\

\underline{Positivity of $\psi$ on $(0,1)$:} \\

Recall \eqref{system_psi}. We have already shown that the function $h$ is decreasing on the interval $I$. In general, to prove that the function $\tilde{A} \cosh(x) - \tilde{B} \sinh(x) > 0$ with $\tilde{A}, \tilde{B} > 0$, it is enough to show that $\tilde{A} - \tilde{B} \geq 0$. That is because it is positive when $\frac{\tilde{A}}{\tilde{B}} > \tanh(x)$ and since $\tanh(x) < 1$ for all $x \in \mathbb{R}$, it is enough if $\frac{\tilde{A}}{\tilde{B}} \geq 1$. Now note that $k(\lambda) := \underbrace{\frac{1+\lambda^2 f'(0)g'(0)}{2\lambda f'(0)} \tanh(1)}_{\tilde{A}} - \underbrace{\lambda g'(0)}_{\tilde{B}}$ can easily be shown to have a unique positive root, $\sqrt{\frac{\tanh(1)}{f'(0)g'(0)[2-\tanh(1)]}}$. Furthermore, $\displaystyle \lim_{\lambda \rightarrow 0^{+}} k(\lambda) = \infty$. Since $k$ is continuous, it must be that $k(\lambda) > 0$ for $\lambda \in \Big(0, \sqrt{\frac{\tanh(1)}{f'(0)g'(0)[2-\tanh(1)]}} \Big)$. Therefore, we must show that 
\begin{align*} 
    \lambda_1 
    < \sqrt{\frac{\tanh(1)}{f'(0)g'(0)[2-
       \tanh(1)]}} 
    < \sqrt{\frac{\frac{2}{\tanh^2(1)}-1}
      {f'(0)g'(0)}}.
\end{align*} 

First, we establish the second part of the inequality. It is true provided that
$\frac{\tanh(1)}{2-\tanh(1)} < \frac{2}{\tanh^2(1)} - 1$. Recall that $\tanh(1) < 1$, therefore $2 - \tanh(1) > 1$. Hence the $\text{LHS} < 1$. Furthermore, $\tanh^2(1) < 1$, hence $\frac{2}{\tanh^2(1)} > 2$ and this implies $\text{RHS} > 1$. Therefore the inequality is satisfied. Now we establish the first part of the previous inequality. We will do so by evaluating $h$ at the positive root of $k$ and showing that we get a negative number [recall that $h(\lambda_1) = 0$ and $h$ achieves its minimum at $\lambda = \sqrt{\frac{\frac{2}{\tanh^2(1)}-1}{f'(0)g'(0)}}$]. Observe that
\small
\begin{align*}
    h \Bigg( \sqrt{\frac{\tanh(1)}{f'(0)g'(0)[2-\tanh(1)]}}\Bigg)
    = [f'(0)g'(0)]^2 \Bigg( \sqrt{\frac{\tanh(1)}{f'(0)g'(0)[2-\tanh(1)]}} \Bigg)^4 + \\
    &\hspace{-2.85in} 
        \Big( 2 - \frac{4}{\tanh^2(1)} \Big) f'(0)g'(0) \Bigg( \sqrt{\frac{\tanh(1)}{f'(0)g'(0)[2-\tanh(1)]}} \Bigg)^2 + 1 \\
    &\hspace{-3.05in} =
        \frac{\tanh^2(1)}{[2-\tanh(1)]^2} +
        2 \Big( \frac{\tanh^2(1)-2}{\tanh^2(1)} \Big) \Big( \frac{\tanh(1)}{2-\tanh(1)} \Big) + 1 \\
    &\hspace{-3.05in} =
         \frac{\tanh^2(1)}{[2-\tanh(1)]^2} + 1
         + 2 \Bigg( \frac{\tanh^2(1)-2}{\tanh(1)[2-\tanh(1)]} \Bigg) \\
    &\hspace{-3.05in} < 0,
\end{align*}
provided that
\begin{align*}
    \Big( \frac{\tanh(1)}{2-\tanh(1)} \Big)^2 + 1
    &< 2 \Bigg( \frac{2-\tanh^2(1)}
         {\tanh(1)[2-\tanh(1)]} \Bigg).
\end{align*}
We establish this inequality by showing the LHS is bounded above by 1 and the RHS is bounded below by 2. For the LHS, $\tanh(1) < 1 \implies 2-\tanh(1) > 1$. Therefore $\frac{1}{2-\tanh(1)} < 1 \implies \frac{\tanh(1)}{2-\tanh(1)} < \tanh(1) \implies \Big( \frac{\tanh(1)}{2-\tanh(1)} \Big)^2 < \tanh^2(1) < 1$. For the RHS, $\frac{2-\tanh^2(1)}{\tanh(1)[2-\tanh(1)]} > 1$ provided $2 - \tanh^2(1) > \tanh(1)[2 - \tanh(1)]$. This is true provided $2 - \tanh^2(1) > 2\tanh(1) - \tanh^2(1)$, or equivalently, if $2 > 2\tanh(1)$. This is immediate again due to $1 > \tanh(1)$. We have now established that $\psi > 0$ on $(0,1)$.


\subsection{Coding Algorithm for Systems}

The tests all use the domain $\Omega = (0,1)$ and analogous discretization parameters.  
First we divide our interval $\Omega = (0,1)$ into $M-1$ subintervals with uniform width $h>0$ letting $x_1 = 0$ and $x_M = 1$. We approximate  $u(x_i)$ by $u_i$ and $v(x_i)$ by $v_i$.  Let $\vec{w} \in \mathbb{R}^{2M}$ be the finite difference solution defined as: 
\[
w_i:=
\begin{cases}
u_i &  \mbox{ for } 1\leq i \leq M,\\
v_{i-M} &  \mbox{ for } M+1 \leq i \leq 2M.
\end{cases}
\]
We similarly use \texttt{fsolve} to solve $F(\vec{w})=\vec{0}$, where $F: \mathbb{R}^{2M}  \to \mathbb{R}^{2M}$ is defined by $F_1(\vec{w}):=\frac{u_1-u_2}{h}-\lambda f (v_1)$; 
$F_i(\vec{w}):=-\frac{1}{h^2}u_{i-1}+(\frac{2}{h^2}+1)u_i-\frac{1}{h^2}u_{i+1}$ for $i \in \{2,3,\ldots , M-1\}$; 
$F_M(\vec{w})=\frac{u_M-u_{M-1}}{h}-\lambda f(v_M)$; 
$F_{M+1}(\vec{w})=\frac{v_1-v_2}{h}-\lambda g (u_1)$;  
$F_i(\vec{w})=-\frac{1}{h^2}v_{i-1-M}+(\frac{2}{h^2}+1)v_{i-M}-\frac{1}{h^2}v_{i+1-M}$ for $i \in \{M+2,M+3,\ldots , 2M-1\}$;  and $F_{2M}(\vec{w})=\frac{v_M-v_{M-1}}{h}-\lambda g(u_M)$.
In all of the numerical experiments, we use the eigenfunction $(\phi, \psi)$ from \eqref{sys_eigenfunctionsAC}, where $\mu_1$ is obtained numerically by solving \eqref{system_mu1}.  
We typically choose a value for $A$ and then use \eqref{system_eigenfunction_AC} to determine $C$ 
instead of solving $\eqref{system_ACsolving1}$ and \eqref{system_ACsolving2} for $A$ and $C$.  
The bifurcation points are obtained by the relationship between $\lambda_1$ and $\mu_1$ from \eqref{system_lambda1}. The exception to using $(\phi, \psi)$ with \eqref{system_eigenfunction_AC} as the initial guess is when the bifurcation point is at $\lambda_1 = \infty$.
Continuation is used to build the bifurcation curve once a single positive solution pair is found for an initial $\lambda$ value.


\subsection{Bifurcation Diagrams and Shape of Solutions for Systems}

We plot solutions of \eqref{d_sol_single} with various $f$ and $\lambda$ values. We see the graphs are concave up with the maximum values along the boundary. We also see that the maximum of the solution decreases as $\lambda$ increases when considering solutions corresponding to the branch of the bifurcation curves bifurcating from $\infty$ at $\lambda = 0$. 
The $\lambda$ value for the initial instance of the problem is chosen to be near but less than the bifurcation point when such a point exists.  
The results for $f(s) = g(s) = s^2$ can be found in Figure~\ref{fig:solutions_quad_quad}, 
$f(s) = s^2 + s$ and $g(s) = s^2$ can be found in Figure~\ref{fig:solutions_quad_quad_linear_without}, 
$f(s) = g(s) = s^2 + s$ can be found in Figure~\ref{fig:quad_quad_linearterms}, 
$f(s) = g(s) = 0.1s - 0.1s^2 + s^3$ can be found in Figure~\ref{fig:cubic_cubic}, 
and $f(s) = 0.1s - 0.1s^2 + s^3$ and $g(s) = s^2$ can be found in Figure~\ref{fig:cubic_quad}.


\begin{figure}[H]
\begin{minipage}{.45\textwidth}
    \includegraphics[scale=0.55]{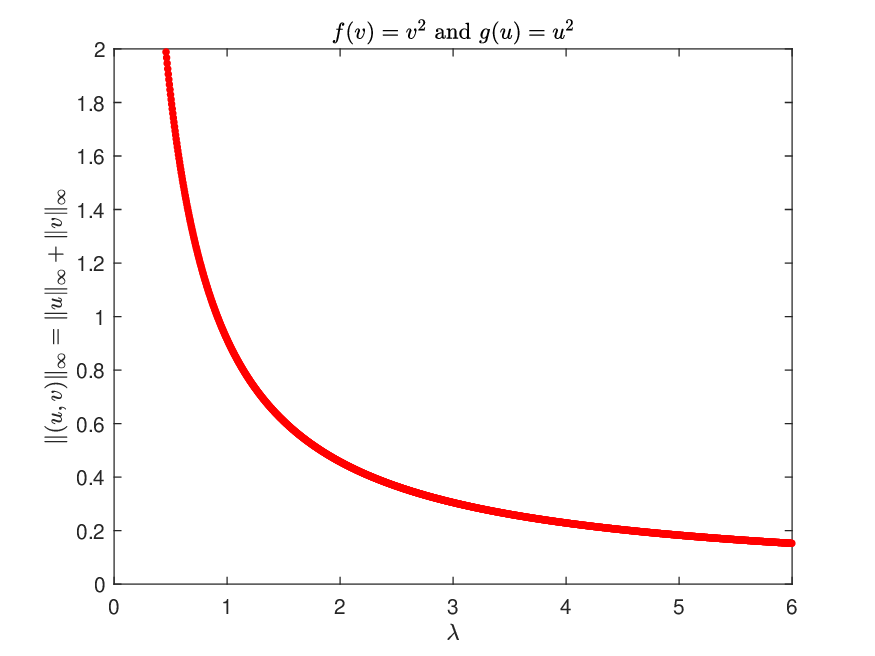}
\end{minipage}
\hspace{.05in}
\begin{minipage}{.45\textwidth}
    \includegraphics[scale=0.48]{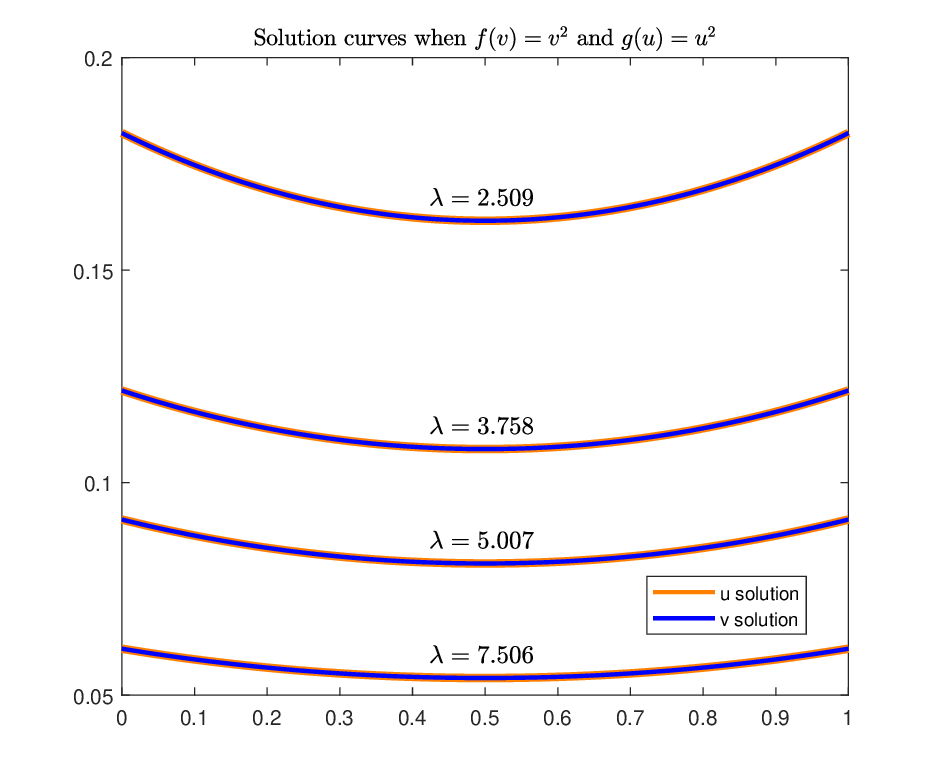}
\end{minipage}
\caption{Bifurcation diagram and sample positive solutions for $f(s) = g(s) = s^2$ using $M=101$. Since $f'(0) = g'(0) = 0$, the location of the bifurcation point is $\lambda_1 = \infty$. We begin by sweeping left along $[0.01, 6]$ using the eigenfunction $(\phi(x), \psi(x)) = (\cosh(x),\sinh(x))$ as the initial guess starting at $\lambda = 6$. The numerical solutions suggest that for all $\lambda \in (0,\lambda_1)$, $u_{\lambda}(x) = v_{\lambda}(x)$ on $[0,1]$. This validates Theorem 1.1 from \cite{BCDMP_Systems}: bifurcation occurs only from infinity.}
\label{fig:solutions_quad_quad}
\end{figure}



\begin{figure}[H]
\begin{minipage}{.45\textwidth}
    \includegraphics[scale=0.52]{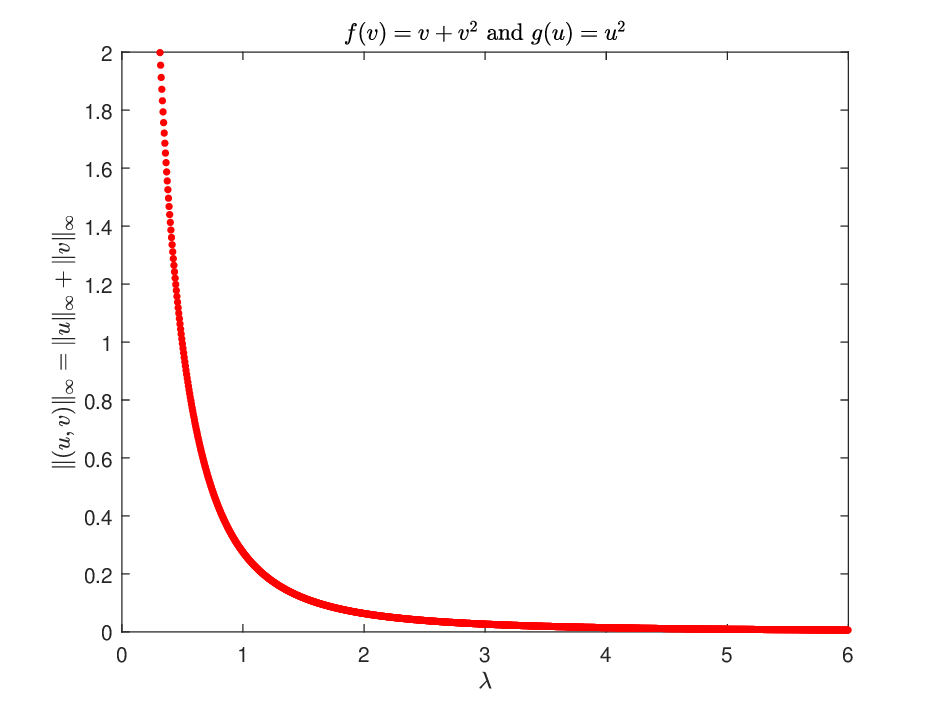}
\end{minipage}
\hspace{.05in}
\begin{minipage}{.45\textwidth}
\includegraphics[scale=0.48]{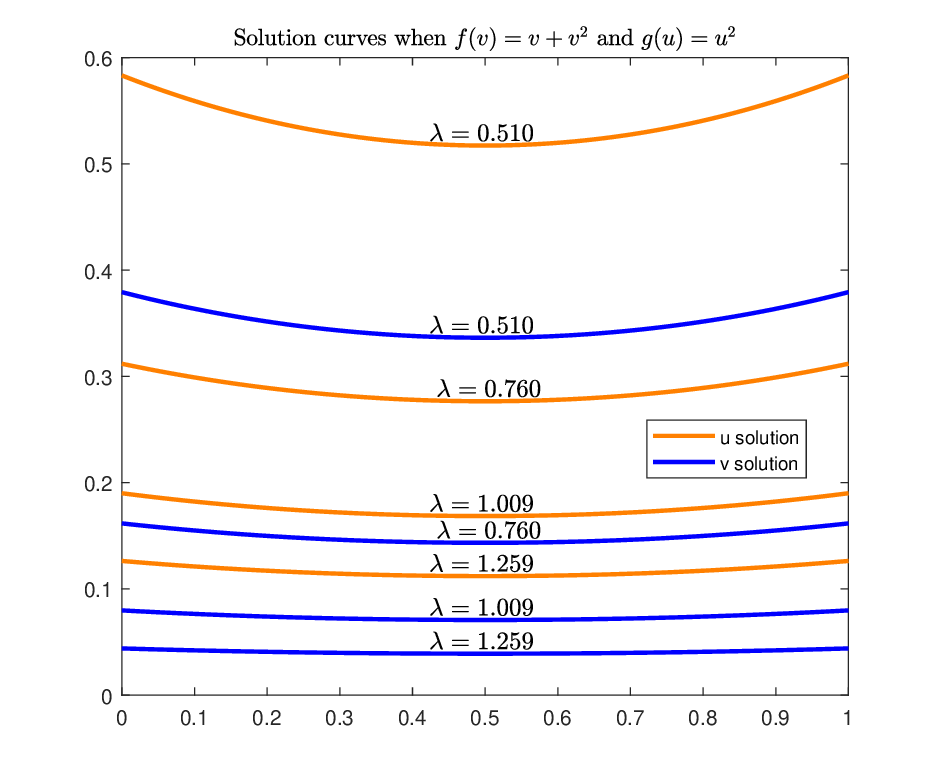}
\end{minipage}
\caption{Bifurcation diagram and sample positive solutions for $f(s) = s^2 + s$, $g(s) = s^2$ using $M=101$. Here $f'(0) = 1$ and $g'(0) = 0$, so by \eqref{system_lambda1}, $\lambda_1 = \infty$ is the location of the bifurcation point. We use the same interval and initial guess as in the previous case. The numerical solutions suggest that for all $\lambda > 0$, we have $v_{\lambda}(x) > u_{\lambda}(x)$ on $[0,1]$. This mixed case still shows bifurcation only from infinity, consistent with the theory.}
\label{fig:solutions_quad_quad_linear_without}
\end{figure}



\begin{figure}[H]
\begin{minipage}{.45\textwidth}
    \includegraphics[scale=0.53]{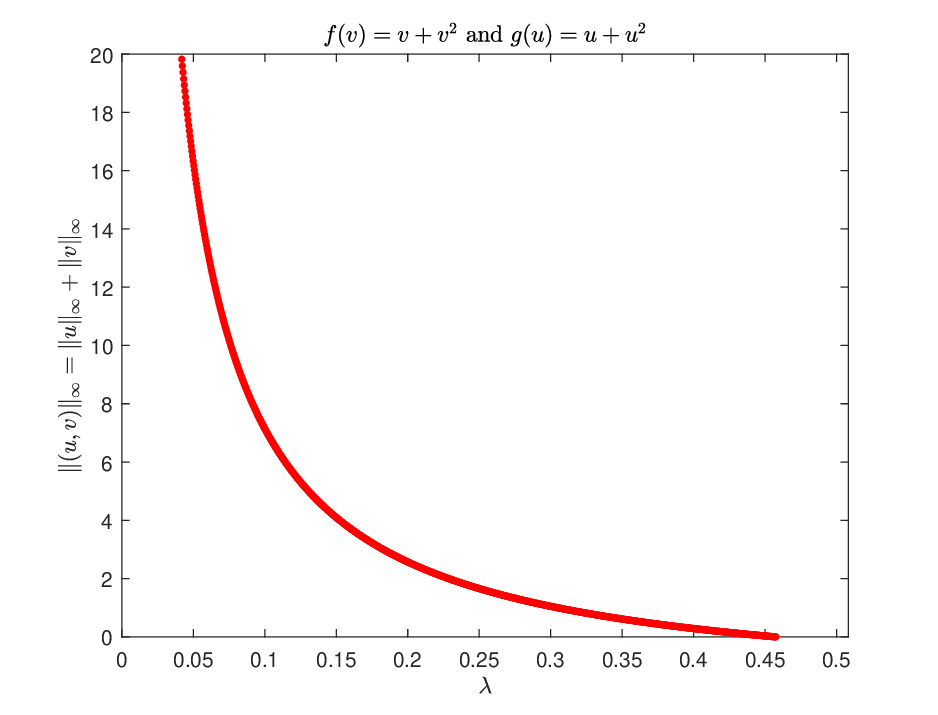}
\end{minipage}
\hspace{.25in}
\begin{minipage}{.45\textwidth}
    \includegraphics[scale=0.45]{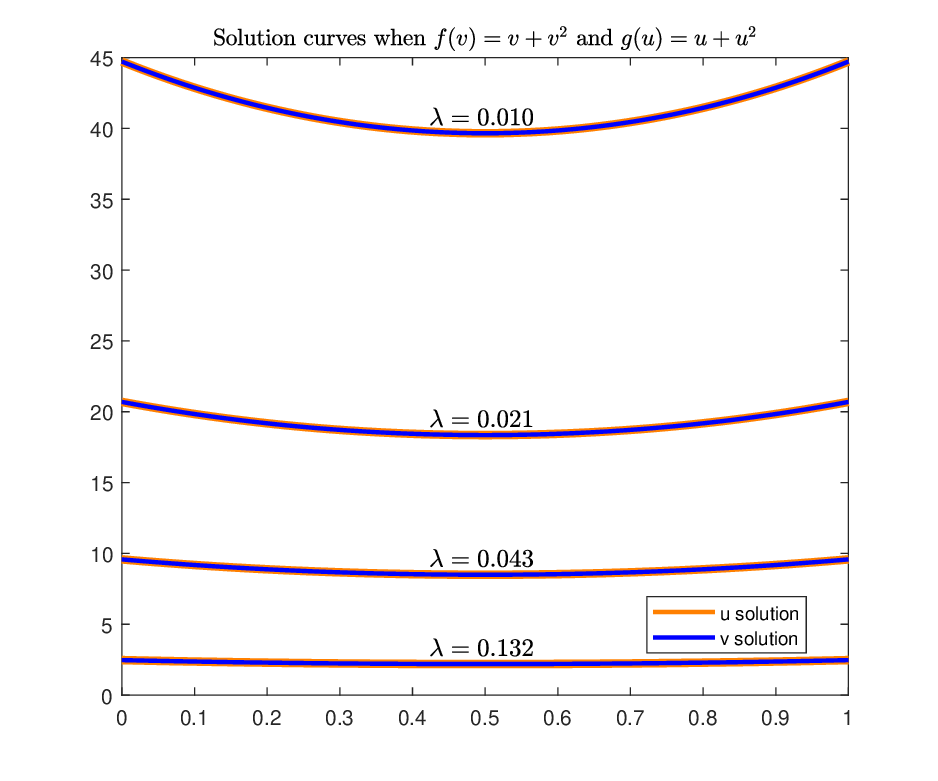}
\end{minipage}
\caption{Bifurcation diagram and sample positive solutions when $f(s) = g(s) = s^2 + s$ using $M=101$. Here $f'(0) = g'(0) = 1$. Hence by \eqref{system_lambda1}, $\lambda_1 \approx 0.46211716$ is the location of the bifurcation point. We begin by sweeping left on $[0,\frac{\lambda_1}{2}]$ by using the eigenfunction $(\phi,\psi)$ with $A = 1$ and $C \approx 1.313035285501163$ (see \eqref{system_eigenfunction_AC}). Then we sweep right along $[\frac{\lambda_1}{2}, \lambda_1]$, initially using the stored solution at $\frac{\lambda_1}{2}$. The numerical solutions suggest that for all $\lambda \in (0,\lambda_1)$, $u_{\lambda}(x) = v_{\lambda}(x)$ on $[0,1]$. This demonstrates Theorem 1.2 from \cite{BCDMP_Systems}: global bifurcation with connected branch from trivial solution to infinity.}
\label{fig:quad_quad_linearterms}
\end{figure}



\begin{figure}[H]
\begin{minipage}{.45\textwidth}
    \includegraphics[scale=0.55]{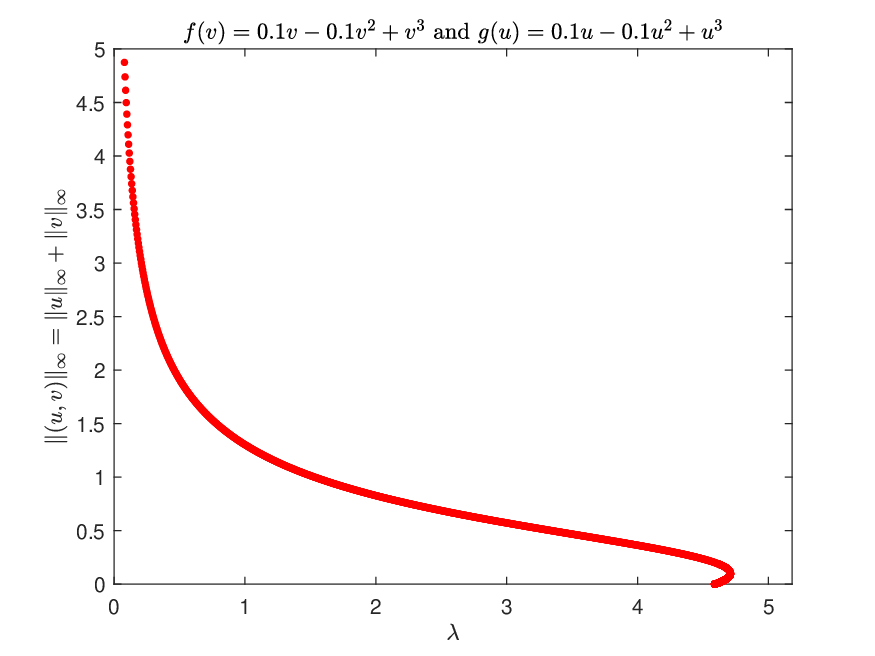}
\end{minipage}
\hspace{.05in}
\begin{minipage}{.45\textwidth}
    \includegraphics[scale=0.45]{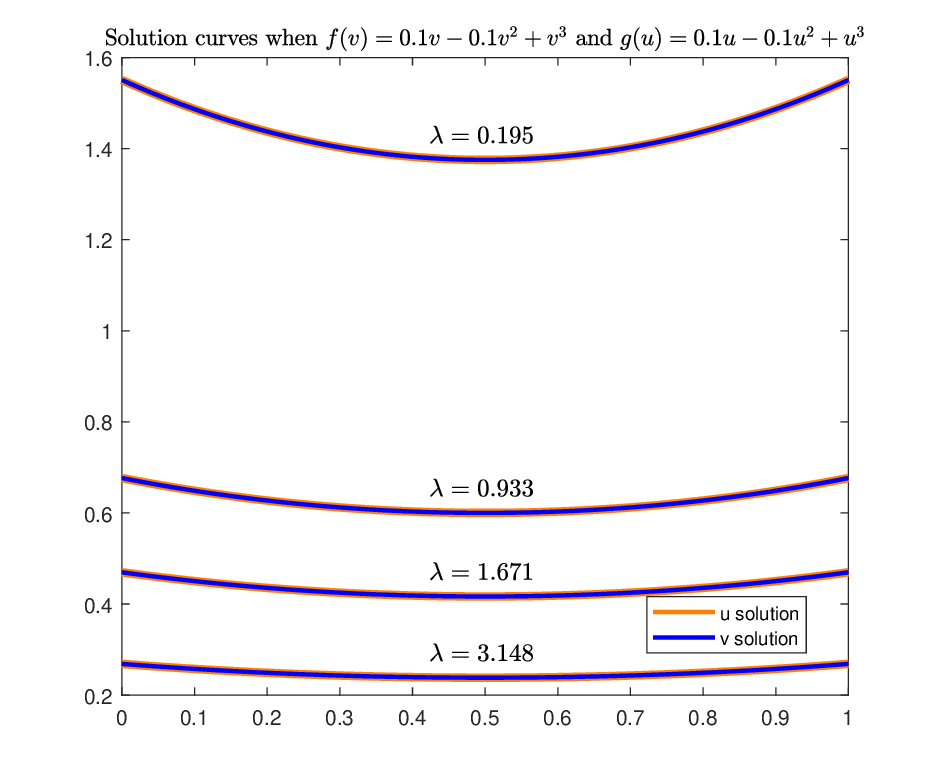}
\end{minipage}
\caption{Bifurcation diagram and sample positive solutions for $f(s) = g(s) = 0.1s - 0.1s^2 + s^3$ using $M=101$. Here $f'(0) = g'(0) = 0.1$. Hence by \eqref{system_lambda1}, $\lambda_1 \approx 4.620403752506588$ is the bifurcation point. Initially, we sweep left along $[0.01,\lambda_1]$ with $A = 1$ and $C \approx 1.313176697711159$ (see \eqref{system_eigenfunction_AC}). Then we sweep right from $[\lambda_1, \lambda_{*}]$, where $\lambda_{*}$ is approximately when \emph{fsolve} first fails to converge. Finally, we sweep left along $[\lambda_1, \lambda_{L}]$, initially using the previously stored solution at $\lambda_{L}$. This demonstrates supercritical bifurcation behavior and validates the multiplicity results from Theorem 1.3 in \cite{BCDMP_Systems}.}
\label{fig:cubic_cubic}
\end{figure}



\begin{figure}[H]
\begin{minipage}{.48\textwidth}
    \includegraphics[scale=0.53]{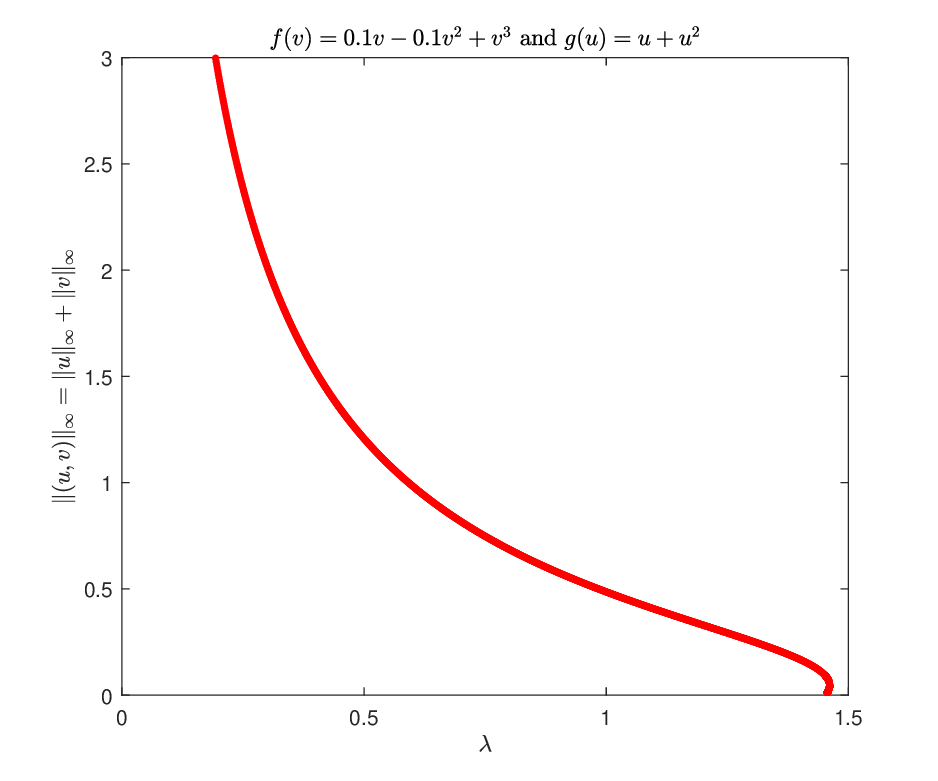}
\end{minipage}
\hspace{.05in}
\begin{minipage}{.45\textwidth}
    \includegraphics[scale=0.50]{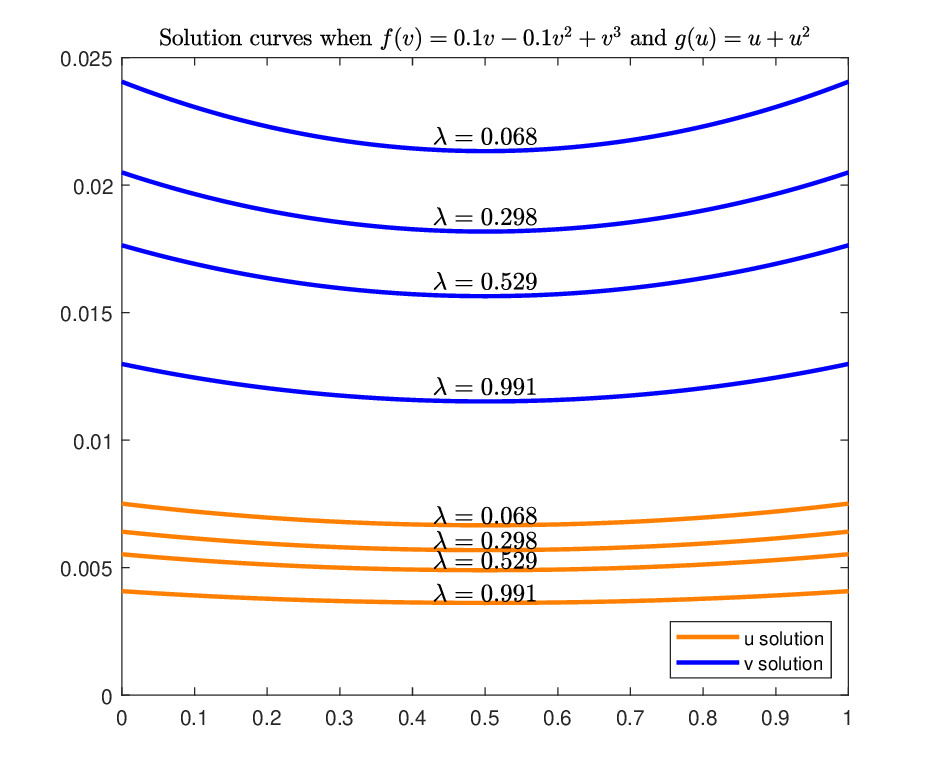}
\end{minipage}
\caption{Bifurcation diagram and solutions for $f(s) = 0.1s - 0.1s^2 + s^3$, $g(s) = s+s^2$ using $M=176$. Here $f'(0) = 0.1$ and $g'(0) = 1$. Hence by \eqref{system_lambda1}, $\lambda_1 \approx 1.46134240$ is the location of the bifurcation point. Initially, we sweep left along the interval $[0.01,\lambda_1]$ with $A = 1$ and $C \approx 4.152182821852064$ (see \eqref{system_eigenfunction_AC}). Then we sweep right from $[\lambda_1, \lambda_{*}]$, where $\lambda_{*}$ is approximately when \emph{fsolve} first fails to converge. Finally, we sweep left along $[\lambda_1, \lambda_{L}]$, initially using the previously stored solution at $\lambda_{L}$. The numerical solutions suggest that for all $\lambda \in (0, \lambda_{*})$, we have $v_{\lambda}(x) > u_{\lambda}(x)$ on $[0,1]$. This mixed case shows intermediate behavior between pure quadratic and pure cubic cases.}
\label{fig:cubic_quad}
\end{figure}


\subsection{Validation of Analytical Results from \cite{BCDMP_Systems}}
\label{subsec:connection-systems}
Our numerical results provide validation for Theorems 1.1, 1.2, and 1.3 corresponding to the theoretical framework for \eqref{PDE_Sys} established in \cite{BCDMP_Systems}, where the problem is transformed into matrix form using a matrix $A$ with eigenvalues $\{\sigma, -\sigma\}$ and $\sigma = \sqrt{f'(0)g'(0)}$. All bifurcation diagrams confirm the existence of connected branches of positive solutions bifurcating from infinity under the superlinear subcritical hypothesis $(H_\infty)$, with solutions approaching infinity as $\lambda \to 0^+$. Our numerical eigenvalue calculations reveal distinct behaviors depending on derivative conditions: when $f'(0) = g'(0) = 0$, we obtain $\lambda_1 = \infty$ (Figures \ref{fig:solutions_quad_quad}, \ref{fig:solutions_quad_quad_linear_without}), while positive derivatives yield finite $\lambda_1$ values matching theoretical predictions and bifurcation from the trivial solution at $\lambda_1 = \frac{\mu_1}{\sqrt{f'(0)g'(0)}}$. The interaction between components $u$ and $v$ creates richer bifurcation structures than single equations, with both components bifurcating simultaneously from infinity and the coupling strength affecting bifurcation point locations. When $R_0 < 0$ (supercritical case), Figure \ref{fig:cubic_cubic} demonstrates the multiplicity phenomenon predicted by Theorem 1.3 in \cite{BCDMP_Systems}: multiple solutions exist in the parameter interval $(\lambda_1, \bar{\lambda})$, with turning point behavior creating regions where exactly two distinct positive solutions coexist.


\section{Acknowledgment} 
The first and second authors were partially supported by the National Science Foundation (NSF) grant DMS-2111059. 
The first author was also supported by the Simons Foundation grant 965180 and the AMS-Simons Travel Grants program.

\bibliographystyle{abbrv}

\begin{thebibliography}{99}

\bibitem{Amann_1976}
H. Amann,
\emph{Nonlinear elliptic equations with nonlinear boundary conditions},
In: New Developments in Differential Equations, North-Holland Math. Studies 21 (1976) 43--63.

\bibitem{Amster-DeNapoli}
P. Amster, P. De Napoli,
\emph{Existence of solutions for elliptic systems with nonlocal terms in one dimension},
J. Math. Anal. Appl. 327 (2) (2007) 1160--1168.

\bibitem{BCDMP_2024}
S. Bandyopadhyay, M. Chhetri, B.B. Delgado, N. Mavinga, R. Pardo,
\emph{Bifurcation and multiplicity results for elliptic problems with subcritical nonlinearity on the boundary},
J. Differential Equations 411 (2024) 28--50.

\bibitem{BCDMP_2022}
S. Bandyopadhyay, M. Chhetri, B.B. Delgado, N. Mavinga, R. Pardo,
\emph{Maximal and minimal weak solutions for elliptic problems with nonlinearity on the boundary},
Electron. Res. Arch. 30 (6) (2022) 2121--2137.

\bibitem{BCDMP_Systems}
S. Bandyopadhyay, M. Chhetri, B.B. Delgado, N. Mavinga, R. Pardo,
\emph{On semilinear elliptic systems with subcritical boundary conditions},
Manuscript submitted for publication.

\bibitem{BLM_2025}
S. Bandyopadhyay, T. Lewis, N. Mavinga, 
\emph{Existence of maximal and minimal weak solutions and finite difference approximations for elliptic systems with nonlinear boundary conditions}, 
Electron. J. Differential Equations 43 (2025) 1--21. 

\bibitem{CantrellCosner_2006}
R.S. Cantrell, C. Cosner,
\emph{Spatial Ecology via Reaction-Diffusion Equations},
Wiley Series in Mathematical and Computational Biology, Chichester, 2003.

\bibitem{Dil_Oth_1999}
R. Dilao, J.R. Ockendon,
\emph{Bifurcation analysis of reaction-diffusion systems with nonlinear boundary conditions},
Physica D 135 (1999) 137--152.

\bibitem{Gordon-Ko-Shivaji}
P. Gordon, E. Ko, R. Shivaji,
\emph{Multiplicity results for classes of singular problems on an exterior domain},
Discrete Contin. Dyn. Syst. 25 (4) (2009) 1417--1428.

\bibitem{Gurung-Gokul-Adhikary}
D.B. Gurung, K.C. Gokul, P.R. Adhikary,
\emph{Mathematical model of thermal effects of blinking in human eye},
Int. J. Biomath. 9 (1) (2016) 1650006.

\bibitem{Inkmann}
F. Inkmann,
\emph{Existence and multiplicity theorems for semilinear elliptic equations with nonlinear boundary conditions},
Indiana Univ. Math. J. 31 (2) (1982) 213--221.

\bibitem{LaceyOckendonSabina_1998}
A.A. Lacey, J.R. Ockendon, J. Sabina,
\emph{Multidimensional reaction diffusion equations with nonlinear boundary conditions},
SIAM J. Appl. Math. 58 (5) (1998) 1622--1647.

\bibitem{Lewis_Morris_Zhang_2022}
T. Lewis, Q. Morris, Y. Zhang,
\emph{Convergence, stability analysis, and solvers for approximating sublinear positone and semipositone boundary value problems using finite difference methods},
J. Comput. Appl. Math. 404 (2022) 113880.

\bibitem{Lewis_Xue_2025}
T. Lewis, X. Xue,
\emph{A high order correction to the Lax--Friedrichs method for approximating stationary Hamilton--Jacobi equations},
arXiv:2502.03728 (2025).

\bibitem{Liu-Shi_2018}
P. Liu, J. Shi,
\emph{Bifurcation of positive solutions to scalar reaction-diffusion equations with nonlinear boundary condition},
J. Differential Equations 264 (1) (2018) 425--454.

\bibitem{Mavinga-Pardo_PRSE_2017}
N. Mavinga, R. Pardo,
\emph{Bifurcation from infinity for reaction-diffusion equations under nonlinear boundary conditions},
Proc. Roy. Soc. Edinburgh Sect. A 147 (3) (2017) 649--671.

\bibitem{Neu-Swift}
J. Neu, D.H. Swift,
\emph{Numerical approximation methods for reaction-diffusion equations},
In: Mathematical Models in Biology, Lecture Notes in Biomathematics 71 (1988) 156--175.

\bibitem{Neu-Swift-Sie}
J. Neu, D.H. Swift, H. Siebert,
\emph{Advanced numerical methods for PDE systems with nonlinear boundary conditions},
J. Comput. Phys. 85 (2) (1989) 374--392.

\bibitem{Pao_book}
C.V. Pao,
\emph{Nonlinear Parabolic and Elliptic Equations},
Plenum Press, New York, 1992.

\bibitem{Pao1}
C.V. Pao,
\emph{Numerical methods for semilinear parabolic equations},
SIAM J. Numer. Anal. 24 (1) (1987) 24--35.

\bibitem{Pao2}
C.V. Pao,
\emph{Finite difference reaction-diffusion systems with coupled boundary conditions and time delays},
J. Math. Anal. Appl. 272 (2) (2002) 407--434.

\bibitem{Pao3}
C.V. Pao,
\emph{Numerical analysis of coupled systems of nonlinear parabolic equations},
SIAM J. Numer. Anal. 36 (2) (1999) 394--416.

\bibitem{Sattinger1973}
D.H. Sattinger,
\emph{Monotone methods in nonlinear elliptic and parabolic boundary value problems},
Indiana University Mathematics Journal \textbf{21} (1973) 979--1000.

\bibitem{Umezu-2023}
K. Umezu,
\emph{Uniqueness of a positive solution for the Laplace equation with indefinite superlinear boundary condition},
J. Differential Equations 350 (2023) 124--151.

\end{thebibliography}

{\footnotesize  
\medskip
\medskip
\vspace*{1mm} 
 
\noindent {\it Shalmali Bandyopadhyay}\\  
Department of Mathematics and Statistics \\ 
The University of Tennessee at Martin\\
Martin, TN\\
E-mail: {\tt sbandyo5@utm.edu}\\ \\  

\noindent {\it Thomas Lewis}\\  
Department of Mathematics and Statistics \\ 
The University of North Carolina at Greensboro \\ 
Greensboro, NC\\
E-mail: {\tt tllewis3@uncg.edu}\\ \\

\noindent {\it Dustin Nichols}\\  
Department of Mathematics \\ 
High Point University\\
High Point, NC\\
E-mail: {\tt dnichols@highpoint.edu}\\ \\

\end{document}